\documentclass[11pt]{amsart}
\usepackage[a4paper,hmargin=3.5cm,vmargin=4cm]{geometry}
\usepackage{graphicx} 
\usepackage{verbatim,enumerate}
\usepackage[dvipsnames]{xcolor}
\usepackage[colorlinks]{hyperref}
\graphicspath{{./Figures/}}
\usepackage {svg}
\usepackage{wrapfig}
\usepackage{amsrefs}

\usepackage{pb-diagram}

\title[Duality of boundary value problems]{%
Duality of boundary value problems for\\ minimal and maximal surfaces
}

\author[S.~Akamine and H.~Fujino]{
S.~Akamine and H.~Fujino
}   

\address[Shintaro Akamine]{%
Graduate School of Mathematics, 
Nagoya University, Chikusa-ku, Nagoya 464-8602, Japan
}
\email{s-akamine@math.nagoya-u.ac.jp}

\address[Hiroki Fujino]{%
  Institute for Advanced Research, Graduate School of Mathematics, 
Nagoya University, Chikusa-ku, Nagoya 464-8602, Japan
}
\email{m12040w@math.nagoya-u.ac.jp}



\thanks{
The first author was partially supported by 
JSPS KAKENHI Grant Number 19K14527, 17H06466 and JSPS/FWF Bilateral Joint Project I3809-N32 ``Geometric Shape Generation'',
and the second author by JSPS KAKENHI Grant Number 19K21022.
}

\usepackage{amsthm}
\theoremstyle{plain}

 \newtheorem{theorem}{Theorem}[section]
 \newtheorem{proposition}[theorem]{Proposition}
 
 \newtheorem{fact}[theorem]{Fact}
 \newtheorem{lemma}[theorem]{Lemma}
 \newtheorem{corollary}[theorem]{Corollary}

\theoremstyle{definition}
 \newtheorem{definition}[theorem]{Definition}
\theoremstyle{remark}
 \newtheorem{remark}[theorem]{Remark}
 \newtheorem*{remark*}{Remark}

 \newtheorem*{acknowledgement}{Acknowledgement}

\numberwithin{equation}{section}

\renewcommand{\phi}{\varphi}

\newcommand{\alim}{\angle \! \lim}

\definecolor{Blue}{rgb}{0,0,1}  
\definecolor{Red}{rgb}{1,0,0}  
\begin{document}
\maketitle

\begin{abstract}
 In 1966, Jenkins and Serrin gave existence and uniqueness results for infinite boundary value problems of minimal surfaces in the Euclidean space, and after that such solutions have been studied by using the univalent harmonic mapping theory. In this paper, we show that there exists a one-to-one correspondence between solutions of infinite boundary value problems for minimal surfaces and those of lightlike line boundary problems for maximal surfaces in the Lorentz-Minkowski spacetime. We also investigate some symmetry relations associated with the above correspondence together with their conjugations, and observe function theoretical aspects of the geometry of these surfaces. Finally, a reflection property along lightlike line segments on boundaries of maximal surfaces is discussed.
\end{abstract}

%============================================================INTRODUCTION====
%\tableofcontents
%=========================MAIN RESULTS===================
\section{Introduction} \label{sec:1} 
Let $\varphi$ be a solution of the minimal surface equation in the Euclidean space $\mathbb{E}^3$
over a simply connected domain $\Omega$ in the plane $\mathbb{R}^2 \simeq \mathbb{C}$,
that is, $\varphi$ is a real-valued function on $\Omega$ whose graph,  denoted by ${\rm graph}(\varphi)$, is a minimal surface.
Similarly, let $\psi$ be a solution of the maximal surface equation over $\Omega$
in the Lorentz-Minkowski space $\mathbb{L}^3$ with signature $(+,+,-)$.
We assume that $\psi$ is the {\it dual} of $\varphi$ in the sense explained later (see Section \ref{sec:2.1}).
Then, we can show that there always exists an orientation-preserving univalent harmonic mapping
$f$ from the unit disk $\mathbb{D}$ onto $\Omega$ such that
$X_{\min}=(f,\varphi \circ f)$ and $X_{\max}=(f, \psi\circ f)$
give global isothermal parameterizations for $\Sigma_{\min}={\rm graph}(\varphi)$ 
and $\Sigma_{\max}={\rm graph}(\psi)$, respectively.
This correspondence gives a quite useful tool to study infinite boundary value problems for minimal graphs and lightlike line boundary value problems
for maximal graphs. These problems have been developed independently; the former was discussed by Jenkins and Serrin in \cite{JS}, the latter was discussed by Bartnik and Simon in \cite{BS} as a special case of more general Dirichlet boundary value problems, respectively.
However, in this paper, we prove the following theorem which enables us to study these boundary value problems simultaneously. 

\begin{theorem}\label{thm:1}
	Let $\Omega$ be a bounded simply connected Jordan domain
	whose boundary contains a line segment $I$. 
	We let $\varphi$ be a solution of the minimal surface equation over $\Omega$, 
	$\psi$ its dual solution of the maximal surface equation,
	and $f\colon \mathbb{D}\rightarrow \Omega$ the corresponding harmonic mapping.
	Then the following statements are equivalent.
	\begin{itemize}
			\item[(i)] $\phi$ tends to plus infinity on $I$, and
			\item[(ii)] $\psi$ tamely degenerates to a future-directed lightlike line segment on $I$.
	\end{itemize} 
	Moreover, in this case, the following condition holds.
	\begin{itemize}
			\item[(iii)] There exists a discontinuous point $w_0\in \partial{\mathbb{D}}$ 
						of the boundary function $\hat{f}\colon \partial{\mathbb{D}}\to \partial{\Omega}$ of 
						$f$ such that $I$ lies on the cluster point set $C(f,w_0)$ of $f$ at $w_0$.
	\end{itemize}
\end{theorem}

\noindent
Here, $C(f,w_0)$ consists of the points $z$ so that $z=\lim_{w_n\to w_0} f(w_n)$ 
for some $w_n\in \mathbb{D}$, and the definition of (ii) is given in Definition \ref{def:lightlike_bdr}, which defines a degeneration of ${\rm graph}(\psi)$ to a lightlike line segment on the boundary with an asymptotic estimate.
It should be remarked that if $\varphi$ tends to plus or minus infinity on a boundary arc $C$ of $\Omega$,
then $C$ must be a line segment (see \cite[p.~102]{O}). The proof of Theorem \ref{thm:1} is given in Section \ref{subsec:proofThm1}, and we discuss when the third condition (iii) conversely implies (i) and (ii) in Section \ref{subsec:converseThm1}.

Several applications can be found from Theorem \ref{thm:1}.
For example, we can prove the following reflection principle. %for maximal graphs.
\begin{theorem} \label{thm:2}
	Under the notations in Theorem \ref{thm:1}, 
	suppose $\partial \Omega$ contains line segments $I_1,I_2$,
	which have a common endpoint $z_0$ with interior angle $\alpha$.
	We also assume that $\varphi$ tends to plus or minus infinity on $I_1$ and $I_2$,
	and that the signs on $I_1$ and $I_2$ differ if $\alpha=\pi$. Then the following 
	statements hold$:$
	
	\begin{itemize}  \setlength{\leftskip}{-3ex}
	\item[(a)]
		The isothermal parameterization $X_{\min}=(f,\varphi \circ f)$ of $\Sigma_{\min}={\rm graph}(\varphi)$
		extends to a generalized minimal surface $\widetilde{X_{\min}}$ beyond a vertical line
		segment $L$ over $z_0$.
		Further, the extended surface is exactly the $\pi$-rotation of
		the original surface $\Sigma_{\min}$ with respect to $L$.	
	
	\item[(b)]
		Similarly, the isothermal parameterization $X_{\max}=(f,\psi \circ f)$ of $\Sigma_{\max}={\rm graph}(\psi)$
		extends to a generalized maximal surface $\widetilde{X_{\max}}$ beyond a shrinking singularity $(z_0, \psi(z_0))$,
		the intersection point of two lightlike line segments.
		Further, the extended surface is exactly the point symmetry to $(z_0, \psi(z_0))$ of
		the original surface $\Sigma_{\max}$.	
	\end{itemize}
\end{theorem}
\noindent
More detailed discussions relating the interior angle $\alpha$ with the boundary behavior of $\phi$ are given in Section \ref{subsec:reflection}.

The next applications concern the conjugate surfaces. It is known that the conjugate 
minimal (resp. maximal) surface is defined for each minimal (resp. maximal) surface,
by replacing each component of the isothermal parametrization with its conjugate harmonic function.
Thus, if we have a minimal graph 
and its isothermal parametrization $X_{\min}$ for instance, 
then we canonically obtain three surfaces; the conjugate minimal surface $X_{\min}^{\ast}$,
the dual maximal surface $X_{\max}$, and the conjugate maximal surface $X_{\max}^{\ast}$ of the dual. 
It should be pointed that the conjugation and the dual operation commute, and are also defined for generalized surfaces. 

Under these two operations, we can find striking relationships between $X_{\min}$, $X_{\max}$, $X_{\min}^{\ast}$ and $X_{\max}^{\ast}$. One is the following symmetry concerning the reflection symmetries, see Figure \ref{zu19}.

\begin{corollary} \label{thm:4}
		Under the assumptions in Theorem \ref{thm:2}, 
		\begin{itemize} \setlength{\leftskip}{-3ex}
			\item $X_{\min}$ admits a vertical segment $L$ over $z_0$, and has
						the line symmetry there.
		\end{itemize}
		Further, the following statements hold at the corresponding part to $L$.
		\begin{itemize} \setlength{\leftskip}{-3ex}
			\item $X_{\max}$ admits a shrinking singularity, and has the point symmetry there.
			\item $X_{\min}^{\ast}$ admits a horizontal geodesic curvature line,
						and has the planar symmetry there.
			\item $X_{\max}^{\ast}$ admits a null curve as a folding singularity, and has the folded symmetry there.
		\end{itemize}
		Here, each of these four surfaces is regarded as the extended surface.
\end{corollary}

The next one is the corresponding result to Corollary \ref{thm:4} at infinity.

\begin{corollary} \label{thm:3}
		Assume that $\partial \Omega$ contains three line segments $I_1,I_2, I_3$, and that $I_j$ and $I_{j+1}$ have a common endpoint $z_j$ with interior angle $\alpha_j\neq \pi$ for $j=1,2$. Further, we suppose that $\phi$ tends to plus infinity on $I_2$, and tends to plus or minus infinity on $I_1$ and $I_3$. Then the following statements hold$:$
	\begin{itemize}   \setlength{\leftskip}{-3ex}
		\item $X_{\min}$ diverges to $+\infty$ over the horizontal segment $I_2=(z_1,z_2)$.
		\item $X_{\max}$ admits a future-directed lightlike line segment over $I_2$.
		\item $X_{\min}^{\ast}$ diverges to a vertical line segment of length $|I_2|=|z_2-z_1|$ at infinity in $(z_1-z_2)$-direction. 
		\item $X_{\max}^{\ast}$ diverges to $-\infty$ at infinity in $(z_1-z_2)$-direction. 
	\end{itemize}
Further, these horizontal segment at infinity, lightlike line segment, vertical segment at infinity and infinite  point at infinity correspond to each other under the conjugation and the dual operation.
\end{corollary}

\begin{figure}[htbp]
       \begin{center} \hspace*{-5ex}
           \includegraphics[scale=0.68]{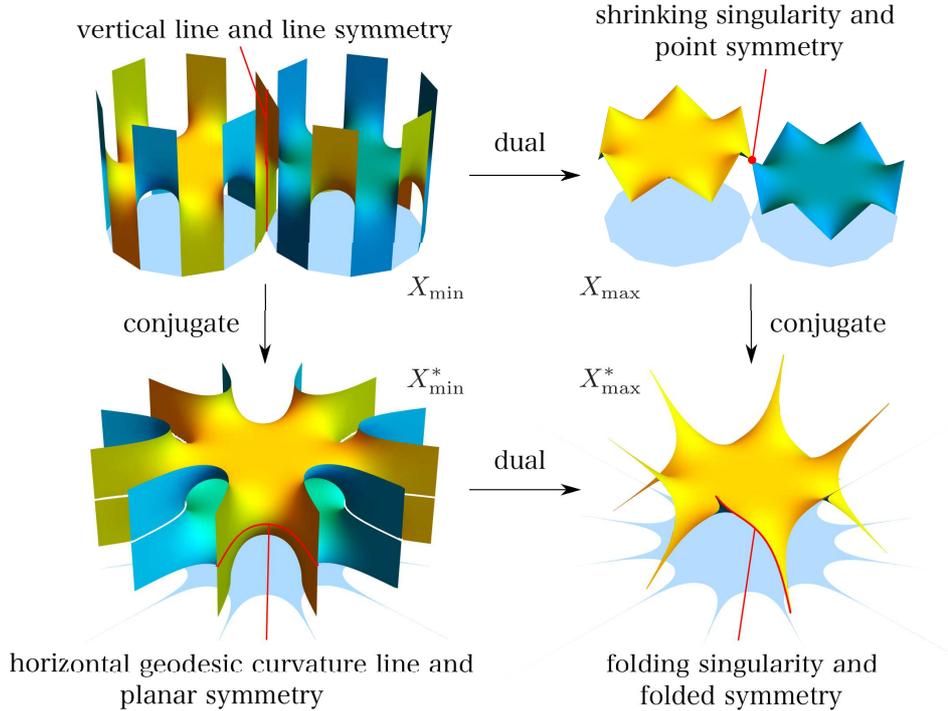} 
       \end{center}
       \caption{Symmetries of $X_{\min}$, $X_{\max}$, $X_{\min}^{\ast}$ and $X_{\max}^{\ast}$ under conjugations and dual operations.}   \label{zu19}
\end{figure}

\noindent
More precise statements of Corollaries \ref{thm:4} and  \ref{thm:3} are given in Section \ref{subsec:singularity} and in Section \ref{subsec:sym}, respectively. The relations between $X_{\min}$ and $X_{\min}^{\ast}$ in Corollaries \ref{thm:4} and  \ref{thm:3} are known as a key tool in the {\it conjugate surface construction} to construct some solutions of (free) boundary value problems (see \cite{Karcher}, \cite{Kar2}, \cite{Kar3}, \cite{Smyth}). Karcher \cite{Karcher} used this relation and constructed ($2k-3$)-families of complete embedded minimal surfaces with vertical translation period, called {\it saddle towers}, from the conjugates of the Jenkins-Serrin graphs in \cite{JS} over equilateral convex $2k$-gons which diverge to plus or minus infinity alternately on each edge.

%=========================HISTORY======================
Infinite boundary value problems for the minimal surface equation have been studied intensively,
and crucial existence and uniqueness results were given by Jenkins and Serrin in \cite{JS}.
Their results can be applied to surprisingly broad situations, however,
we briefly restrict ourselves to the case where the domain $\Omega$ is a
polygonal domain and the prescribed boundary value is plus or minus infinity on each edge.
In this case, it is known that the corresponding harmonic mapping can be written
as the Poisson integral of some step function, and
this fact leads us to more detailed analysis of the solution.
We refer the readers to the references \cite{BW}, \cite{D}, \cite{MS}, \cite{W}.
Further, related deep results on the univalent harmonic mappings can be found
in \cite{BH}, \cite{HS}, \cite{HS2}, for instance.

Meanwhile, boundary value problems for the maximal surface equation
was discussed by Bartnik and Simon in slightly different settings in \cite{BS}.
More precisely, they considered the variational problem of maximizing the
surface area functional among weakly spacelike graphs in general dimensions,
and gave a necessary and sufficient condition for the boundary value problems
to be solvable, together with the uniqueness result.
We remark that the maximal surface equation appears as the
Euler-Lagrange equation of this variational problem when the surface is spacelike.
On the other hand, in a special case of lightlike boundary value problems,
we can also obtain in Corollary \ref{cor:JS} the corresponding result to the Jenkins-Serrin result for maximal surfaces
as an immediate corollary of Theorem \ref{thm:1}.

Finally, we note that the dual correspondence between solutions of the minimal surface equation and the maximal surface equation,
which is one of the main tools in the present article, has appeared in various contexts
not only in the fields of mathematics but also physics.
Here, the duality was established by Calabi \cite{C} in the Lorentzian geometrical setting to study global behavior of maximal surfaces. Also, it played an important role in the arguments by Jenkins and Serrin in \cite{JS}. A fluid mechanical viewpoint of the duality was discussed in \cite{AUY} (cf.\cite{B}). 
Recently, this duality is generalized to surfaces with constant mean curvature in some Riemannian and Lorentzian homogeneous spaces by Lee \cite{L} and more general situation with prescribed mean curvature by Lee and Manzano \cite{LM}.

%============================================================PRELIMINARY====

\section{Preliminaries} \label{sec:2} 

We denote the Euclidean 3-space by $\mathbb{E}^3$ and the Lorentz-Minkowski 3-space with signature $(+,+,-)$ by $\mathbb{L}^3$. Let $(x,y,t)$ be the canonical coordinate on $\mathbb{R}^3$. We sometimes identify $\mathbb{E}^3$ and $\mathbb{L}^3$ with $\mathbb{R}^3$ as real vector spaces, and also identify the  $xy$-plane with the complex plane $\mathbb{C}$, respectively.

%==========================DUALITY============================
\subsection{Duality between minimal surfaces and maximal surfaces}\label{sec:2.1} 
One of the key tools in the present article is the duality between minimal and maximal graphs. We first define the duality.

Let $\phi$ be a solution of the minimal surface equation
\begin{equation}\label{eq:minimal}
		\mathrm{div}\left(\frac{\nabla{\phi}}{\sqrt{1+|\nabla\phi|^2}}\right)=0, \nonumber
\end{equation}
over a simply connected domain $\Omega$ in the $xy$-plane.
Then we can define a function $\psi$ over $\Omega$ such that
\begin{equation}\label{eq:duality}
		d{\psi}=-\frac{\phi_{y}}{\sqrt{1+|\nabla{\phi}|^2}}dx+\frac{\phi_{x}}{\sqrt{1+|\nabla{\phi}|^2}}dy.
\end{equation}
It can be easily seen that $\psi$ is a solution of the maximal surface equation
\begin{equation}\label{eq:maximal}
		\mathrm{div}\left(\frac{\nabla{\psi}}{\sqrt{1-|\nabla\psi|^2}}\right)=0,\quad |\nabla{\psi}|<1. \nonumber
\end{equation}
The following duality among solutions of the minimal surface equation and the maximal surface equation is stated by Calabi in \cite{C}.
\begin{fact}\label{fact:duality}
On a simply connected domain $\Omega\subset \mathbb{R}^2$, up to an additive constant, there is a one-to-one correspondence between solutions $\phi$ of the minimal surface equation and $\psi$ of the maximal surface equation via the relation $(\ref{eq:duality})$.
\end{fact}
\noindent Henceforth we shall call the above $\psi$ satisfying $(\ref{eq:duality})$ the {\it dual} of $\phi$.

%==========================REP. FORMULA============================
\subsection{Minimal surfaces, maximal surfaces and harmonic mappings}

We next recall parametric and non-parametric representations of minimal and maximal surfaces and their relations to harmonic mappings. 

Let us consider the minimal graph in $\mathbb{E}^3$ of a solution $\phi$ of the minimal surface equation, denoted by ${\rm graph}(\phi)$, over a simply connected domain $\Omega\subset \mathbb{C}$. By the uniformization theorem, there exists a global isothermal coordinate $(\mathbb{D}; w=u+iv)$ and a parametrization $X_{\rm min}(w)=(x(w),y(w),t(w))$ on $\mathbb{D}$, so that $X_{\rm min}(\mathbb{D})={\rm graph}(\phi)$. Since each of the coordinate functions $x,y,t$ is harmonic, we obtain a univalent harmonic mapping $f=x+iy$, which gives a diffeomorphism from $\mathbb{D}$ onto $\Omega$ and $X_{\min}=(f,\phi\circ f)$ gives an isothermal parametrization of ${\rm graph}(\phi)$. Further, we can always assume that $f$ is orientation-preserving, by changing $w$ to $\overline{w}$ if necessary. Such $f$ is unique up to a pre-composition with a M\"obius transformation of $\mathbb{D}$, since $X_{\min}=(f,\phi\circ f)$ is a conformal mapping. If we use the canonical decomposition $f=h+\overline{g}$, where $h$ and $g$ are holomorphic functions in $\mathbb{D}$, then the conformality condition of $X_{\rm min}(w)=(x(w),y(w),t(w))$ implies
\[
		0=(\partial{x}/\partial{w})^2+(\partial{y}/\partial{w})^2+(\partial{t}/\partial{w})^2=h'g'+t_w^2. 
\]
Thus there exists a single-valued holomorphic branch of $\sqrt{h'g'}$ so that the third coordinate $t=\phi \circ f$ is written as $t(w)={\rm Re} F(w)$, where 
\begin{equation} \label{eq:FforPhi}
		F(w)=2i\int_0^w{\sqrt{h'(\zeta)g'(\zeta)}}d{\zeta}+\varphi\circ f(0).
\end{equation}
Therefore, ${\rm graph}(\phi)$ has the following parametric representation on $\mathbb{D}$.
\begin{equation}\label{eq:Xmin}
		X_{\rm min}=({\rm Re}(f), {\rm Im}(f), {\rm Re}(F))=(f, {\rm Re}(F)).
\end{equation}
In the last equality, we identify $\mathbb{E}^3$ with $\mathbb{C}\times \mathbb{R}$.

Similarly, for a solution $\psi$ of the maximal surface equation
over $\Omega$, its graph denoted by ${\rm graph}(\psi)$ has the representation
\begin{equation}\label{eq:Xmax}
		X_{\rm max}=({\rm Re}(f), {\rm Im}(f), {\rm Im}(F))= (f, {\rm Im}(F)),
\end{equation}
by using a univalent harmonic mapping $f=h+\overline{g}$ and a holomorphic function $F(w)$ similarly defined by (\ref{eq:FforPhi}). 

Needless to say, the harmonic mapping $f$ in $(\ref{eq:Xmax})$ might be different from one in $(\ref{eq:Xmin})$. However, the following statements guarantee that we can use the same $f$ simultaneously in $(\ref{eq:Xmin})$ and $(\ref{eq:Xmax})$ if $\phi$ and $\psi$ satisfy the duality relation $(\ref{eq:duality})$. 

\begin{proposition}\label{prop:parametric_duality}
Let $f=h+\overline{g}$ be a univalent harmonic mapping from $\mathbb{D}$ onto $\Omega$ such that $h'g'$ has a holomorphic square root $\sqrt{h'g'}$ on $\mathbb{D}$. If we put $F$ as 
\begin{equation} \label{eq:F}
		F(w)=F_f(w)=2i\int_0^w{\sqrt{h'g'}}d{\zeta},
\end{equation}
then the following statements hold.
\begin{itemize}
\item[(i)] The function $\phi={\rm Re}(F)\circ f^{-1}$ gives a minimal graph in $\mathbb{E}^3$, and $X_{\rm min}$ as in \eqref{eq:Xmin} is an isothermal parametrization of ${\rm graph}(\phi)$.
\item[(ii)] The function $\psi={\rm Im}(F)\circ f^{-1}$ gives a maximal graph in $\mathbb{L}^3$, and $X_{\rm max}$ as in \eqref{eq:Xmax} is an isothermal parametrization of ${\rm graph}(\psi)$.
\item[(iii)] If $f$ is orientation-preserving, then $\phi$ in $({\rm i})$ and $\psi$ in $({\rm ii})$ satisfy the duality relation \eqref{eq:duality}. Conversely, if $f$ is orientation-reversing, then $-\psi$ is the dual of $\phi$.
\end{itemize}
\end{proposition} 
\noindent
For a harmonic mapping $f=h+\overline{g}$, the quantity $\omega=\overline{f_{\overline{w}}}/f_w=g'/h'$ is called the {\it analytic dilatation} (or the {\it second Beltrami coefficient}) of $f$. By using it, we can see that $f$ is orientation-preserving (resp.~orientation-reversing) if and only if $|\omega|<1$ (resp.~$|\omega|>1$) and $h'g'$ has a holomorphic square root if and only if so does $\omega$.

\begin{proof}
We can easily prove (i) and (ii) by a similar argument in \cite[Section 10.2]{D}. Then, we here give a proof of (iii).  Assume that $f$ preserves the orientation.
By \eqref{eq:Xmin}, the upward unit normal vector $n_{\text{min}}$ of ${\rm graph}(\phi)$ has the two kinds of representations as follows.
\begin{align}
	\begin{aligned}
		n_{\text{min}}(z)&=\frac{1}{\sqrt{1+|\nabla{\phi}(z)|^2}}(-\phi_{x}(z),-\phi_{y}(z),1) \\ 
		&=\frac{1}{1+|\omega(w)|}(2\mathrm{Im}{\sqrt{\omega(w)}}, 2\mathrm{Re}{\sqrt{\omega(w)}},1-|\omega(w)|), \label{eq:n_min}
	\end{aligned}
\end{align}
where $z\in \Omega$ and $w\in \mathbb{D}$ are related by $z=f(w)$, and $\sqrt{\omega}$ is a holomophic square root of $\omega$ such that $\sqrt{\omega}h'=\sqrt{h'g'}$. Similarly, by \eqref{eq:Xmax},  the future-directed unit normal vector $n_{\text{max}}$ of ${\rm graph}(\psi)$ has the representations
\begin{align}
	\begin{aligned}
		n_{\text{max}}(z)&=\frac{1}{\sqrt{1-|\nabla{\psi}(z)|^2}}(\psi_{x}(z),\psi_{y}(z),1) \\ 
		&=\frac{1}{1-|\omega(w)|}(2\mathrm{Re}{\sqrt{\omega(w)}}, -2\mathrm{Im}{\sqrt{\omega(w)}},1+|\omega(w)|). \label{eq:n_max}
	\end{aligned}
\end{align}
By comparing \eqref{eq:n_min} and \eqref{eq:n_max}, we conclude that $\phi$ and $\psi$ are related by \eqref{eq:duality}. Conversely, if $f$ is orientation-reversing, then $|\omega|>1$. Therefore the equations (\ref{eq:n_min}) and (\ref{eq:n_max}) hold if we multiply the last representations by $-1$, respectively. Similarly, We have the conclusion.
\end{proof}

\begin{remark} 
The statement (iii) of Proposition \ref{prop:parametric_duality} gives another proof of \cite[Theorem 1]{L}, and reveal that the duality \eqref{eq:duality} is nothing but the transformation of minimal and maximal surfaces considered in various situations as in \cite{AL}, \cite{LLS}, \cite{UY1}, in addition to the situations discussed in Introduction.
\end{remark}

%==========================HARMONIC MAPPINGS======================

%==========================GENERALIZED SURFACES====================
\subsection{Generalized minimal and maximal surfaces}

To deal with singularities on minimal and maximal surfaces, we recall the classes of generalized minimal surfaces (see \cite[p.~47]{O}) and generalized maximal surfaces introduced in \cite{EL}.  

Let $X$ be a non-constant harmonic mapping from a Riemann surface $M$ to $\mathbb{E}^3$ (resp.~$\mathbb{L}^3$). Suppose that at any point $p\in M$ there exists a complex coordinate neighborhood $(U; w=u+iv)$ such that the derivatives $\Phi=(\Phi_1,\Phi_2,\Phi_3)=\partial{X}/\partial{w}=\left(\partial{x}/\partial{w}, \partial{y}/\partial{w}, \partial{t}/\partial{w}\right)$ satisfy
 \begin{eqnarray*}
 	\begin{gathered}
		(\Phi_1)^2+(\Phi_2)^2+(\Phi_3)^2=0\\
		\text{(resp.~$(\Phi_1)^2+(\Phi_2)^2-(\Phi_3)^2=0$ and $|\Phi_1|^2+|\Phi_2|^2-|\Phi_3|^2\not \equiv 0$)}. 
	\end{gathered}
\end{eqnarray*}
Then $X$ is said to be a {\it generalized minimal surface} (resp.~ {\it a generalized maximal surface}). 

We remark that for a generalized minimal surface $X$, the condition $|\Phi_1|^2+|\Phi_2|^2+|\Phi_3|^2\not \equiv 0$ holds automatically by the non-constancy of $X$. Further, a point on $M$ at which $X$ satisfies $|\Phi_1|^2+|\Phi_2|^2+|\Phi_3|^2= 0$ is called a {\it branch point} of $X$. 

On the other hand, for a generalized maximal surface $X$, the set of points on $U$ on which $|\Phi_1|^2+|\Phi_2|^2-|\Phi_3|^2=0$ is divided into 
\[
		\mathcal{A}=\{p\in U \mid \text{$X_u(p)$ or $X_v(p)$ is lightlike in $\mathbb{L}^3$} \},\quad \mathcal{B}=\{p\in U\mid d{X}_p=0\}.
\]
We call a point $p$ in $\mathcal{A}\cup \mathcal{B}$ a {\it singular point} of $X$ and, in particular, it is called a {\it branch point} of $X$ if $p\in \mathcal{B}$. For generalized maximal surfaces, Kim and Yang \cite{KY} introduced two kinds of important singular points as follows: A singular point $p \in \mathcal{A}$ is called a {\it shrinking singular point} (or a {\it conelike singular point}) if there is a neighborhood $U$ of $p$ and a regular curve $\gamma\colon I \to U$ from an interval $I$ such that $\gamma(I)\subset \mathcal{A}$ and $X\circ \gamma (I)$ becomes a point in $\mathbb{L}^3$, which we call a {\it shrinking singularity}.
Also a singular point $p \in \mathcal{A}$ is called a {\it folding singular point} (or a {\it fold singular point}) if there is an isothermal coordinate system $(U; u,v)$ such that $p=(0,0)$ and $X_v(u,0)\equiv 0$. We call the image $\{X(u,0)\mid (u,0)\in U)\}$ a {\it folding singularity}. By definition, the curve $\gamma(u)=X(u,0)$ representing the folding singularity is a null curve, which is a curve whose velocity vector field is lightlike.

%============================================================SECTION 3========
\section{Duality of boundry value problems} \label{sec:3} 

%==========================PROOF of THM1============================
\subsection{Proof of Theorem \ref{thm:1}} \label{subsec:proofThm1}
Throughout this paper, we assume that a segment is open unless otherwise noted. At first, we introduce the following concept:

\begin{definition}\label{def:lightlike_bdr}
Let $\Omega \subset \mathbb{C}$ be a Jordan domain and $I\subset \partial{\Omega}$ an open line segment with outward unit normal $\nu$. 
We say that a solution $\psi$ of the maximal surface equation over $\Omega$ {\it tamely degenerates to a future-directed lightlike line segment on $I$}, if $\psi$ satisfies 
 \begin{equation}\label{eq:lightlike_bdr_condition}
 \frac{\partial{\psi}}{\partial{\tau}}(z)=1+O(\mathrm{dist}(z,J)^2) \text{ as } z\to J  \nonumber
 \end{equation}
for each closed segment $J\subset I$, where $\partial/\partial \tau$ denotes the directional derivative in the direction $\tau=i\nu$. When $-\psi$ tamely degenerates to a future-directed lightlike line segment, we say that $\psi$ {\it tamely degenerates to a past-directed lightlike line segment}. Moreover, we simply say that $\psi$ {\it tamely degenerates to a lightlike line segment} if it tamely degenerates to a future or past-directed lightlike line segment.
\end{definition}

It can by easily seen that the following statements hold, by definition. According to this fact, the term ``tamely degenerate'' actually defines a  degeneration to a lightlike line segment on the boundary with an asymptotic estimate. 

\begin{proposition}
	Assume that $\psi$ tamely degenerates to a future-directed lightlike line segment on $I\subset \partial \Omega$. Then, 
	\begin{itemize}
		\item $\psi|_I$ parametrizes a lightlike line segment, which is future-directed with respect to the positive orientation on $I\subset \partial \Omega$.
		\item The first fundamental form $ds^2$ of ${\rm graph}(\psi)$ degenerates on $I$:
		
		\noindent
		It holds that $\det (ds_z^2)=1-|\nabla\psi(z)|^2 \to 0$ as $z\to I$.
	\end{itemize}
\end{proposition}

\begin{remark}
Under the notations in Definition \ref{def:lightlike_bdr}, if we assume that $\psi$ is $C^2$-differentiable on $\Omega\cup I$ and ${\rm graph}(\psi)$ has a lightlike line segment $L$ over $I$, then $\psi$ tamely degenerates to $L$ automatically. 
Under this assumption, the statement (i) in Theorem \ref{thm:1} was proved in \cite{AUY} and it played an important role to prove an improvement of the Bernstein-type theorem in $\mathbb{L}^3$.  
\end{remark}

To prove Theorem \ref{thm:1}, we recall the following Hengartner-Schober's result \cite[Theorem 4.3]{HS} (cf.~\cite[page 35]{D})

\begin{lemma}\label{lemma:HS}
Let $f$ be an orientation-preserving univalent harmonic mapping from $\mathbb{D}$ into a bounded Jordan domain $\Omega$. Suppose the radial limit $\lim_{r\to 1}{f(re^{i\theta})}$ exists and belongs to $\partial{\Omega}$ for almost every $\theta$. Then there exists a countable set $E\subset \partial{\mathbb{D}}$ which satisfies the following:
\begin{itemize}
\item[(i)] For each $e^{i\theta}\in \partial \mathbb{D} \setminus E$, the unrestricted limit $\hat{f}(e^{i\theta})=\lim_{w\to e^{i\theta}}{f(w)}$ exists and belongs to $\partial \Omega$. Further $\hat{f}$ is continuous on $\partial{\mathbb{D}}\setminus E$.
\item[(ii)] The one-sided limits
\[
\hat{f}{(e^{i\theta_+})}=\lim_{t \to \theta^+}{\hat{f}(e^{it})},\quad \hat{f}{(e^{i\theta_-})}=\lim_{t \to \theta^-}{\hat{f}(e^{it})}
\]
exist, belong to $\partial{\Omega}$ and are different for each $e^{i\theta}\in E$. Here the limits are taken on $\partial \mathbb{D} \setminus E$.
\item[(iii)] The cluster set $C(f,e^{i\theta})$ of $f$ at $e^{i\theta}\in E$ is the closed line segment joining $\hat{f}{(e^{i\theta_+})}$ and $\hat{f}{(e^{i\theta_-})}$.
\end{itemize}
\end{lemma}

\noindent
We should remark that since a univalent harmonic mapping onto a bounded Jordan domain always satisfies the assumption of Lemma \ref{lemma:HS} (cf.~\cite[p.5]{D}), we can take such a countable set $E$ of the discontinuous points, and $f$ can be written as the Poisson integral of its boundary function $\hat{f}$.\\

%==========================PROOF============================
\noindent
{\bf Proof of Theorem \ref{thm:1}}.
The following proof of (i) $\Rightarrow$ (ii) is based on a standard argument by using the estimate in \cite[Lemma 1]{JS}, and (i) $\Rightarrow$ (iii) is  essentially given by Bshouty and Weitsman in \cite[Theorem 1]{BW}, however, we give here the proofs of these parts for the sake of completeness and since the settings are slightly different.

We may assume that $I=(a,b)$, ($a<b$) and $\Omega$ lies in the upper half-plane $\mathbb{H}$ along $I$. Henceforth, we consider the problems under this situation.

First we assume that $\phi \to \infty$ as $z\to I$, and prove that $\psi_x(z)=1+O(\mathrm{dist}(z,J)^2)$ as $z\to J$ for every closed segment $J\subset I$. For arbitrary $x_0\in I$, we take $\varepsilon >0$ such that $D(x_0,9\varepsilon) \cap (\partial{\Omega}\setminus I)=\emptyset$, where $D(x_0, R)$ denotes the open disk centered at $x_0$ with radius $R$. 
We set $J_0=(x_0-\varepsilon, x_0+\varepsilon)$ and $D=D(x_0,9\varepsilon)\cap \mathbb{H}$.
If $z\in D(x_0,\varepsilon)\cap \mathbb{H}$, then
\[
d:=\mathrm{dist}(z,\partial{D})=\mathrm{dist}(z,J_0)=\mathrm{Im}{z}<\varepsilon.
\]
Moreover, if we set $\Sigma_{\min}={\rm graph}(\phi)$ and $\Sigma'=\Sigma_{\text{min}}|_D$, then the geodesic distance $r$ from $(z,\phi(z))\in \Sigma'$ to $\partial{\Sigma'}$ satisfies
\begin{equation}\label{eq:8epsilon}
r\geq \mathrm{dist}(z,\partial{D}\setminus I)>8\varepsilon,
\end{equation}
by the assumption $\phi(z)\to \infty$ ($z\to I$). Thus we can apply \cite[Lemma 1]{JS} to the convex domain $D$ since $d<r/8$, and we have
\[
1>\frac{|\phi_y(z)|}{\sqrt{1+|\nabla{\phi}|^2}} \geq 1-4\frac{d^2}{r^2}.
\]
In particular, the sign of the continuous function $\phi_y$ does not change on $D(x_0,\varepsilon)\cap \mathbb{H}$. Taking into account the assumption $\phi(z)\to \infty$ ($z\to I$), we have $|\phi_y(z)|=-\phi_y$ and 
\[
1>\psi_x(z)\geq 1-4\frac{d^2}{r^2}>1-\frac{1}{16\varepsilon^2}\mathrm{dist}(z,J_0)^2,
\]
by the duality \eqref{eq:duality} and \eqref{eq:8epsilon}. Thus we have $|1-\psi_x(z)|<C\mathrm{dist}{(z,J_0)}^2$ for $C=1/(16\varepsilon^2)$. Since each closed segment $J\subset I$ is covered by a finite number of such $J_0$, we obtain the desired estimate.

Conversely, let us assume that $\psi$ tamely degenerates to a future-directed lightlike line segment on $I$. We prove that $\phi(z_n)\to \infty$ for each sequence $\{z_n\}_n$ in $\Omega$ which converges to $z_0\in I$.
If we take $\varepsilon>0$ sufficiently small, then $R(z_0,\varepsilon)\cap (\partial{\Omega}\setminus I)=\emptyset$, where  $R(z_0,\varepsilon)=\{z\in \mathbb{C}\mid |\mathrm{Re}{z}-\mathrm{Re}{z_0}|\leq \varepsilon,  |\mathrm{Im}{z}-\mathrm{Im}{z_0}|\leq \varepsilon\}$, and there exists $C'>0$ such that
\[
|1-\psi_x(z)|\leq C'|\mathrm{Im}{z}|^2
\]
holds for $z\in R(z_0,\varepsilon) \cap \Omega$.
Then the inequality $1-|\nabla{\psi}|^2\leq (1+\psi_x)(1-\psi_x)<2C'|\mathrm{Im}{z}|^2$ holds. Therefore, we obtain
\[
-\phi_y=\frac{\psi_x}{\sqrt{1-|\nabla{\psi}|^2}} \geq\frac{1-C'|\mathrm{Im}z|^2}{\sqrt{2C'}|\mathrm{Im}z|} \geq \frac{C_1}{|\mathrm{Im}z|}-C_2
\] 
for some $C_1, C_2>0$. Without loss of generality, we may suppose that the sequence $\{z_n\}_n$ is in $R(z_0,\varepsilon) \cap \Omega$
and hence $z_n=x_n+iy_n$ satisfies
\begin{align*}
\phi(z_n)-\phi(z_1)&=\displaystyle \int_{x_1}^{x_n}\phi_x(x+iy_1)dx+\displaystyle \int_{y_1}^{y_n}\phi_y(x_n+iy)dy\\
&\geq -\displaystyle \int_{\mathrm{Re}{z_0}-\varepsilon}^{\mathrm{Re}{z_0}+\varepsilon}|\phi_x(x+iy_1)|dx + \displaystyle \int_{y_n}^{y_1}\left(\frac{C_1}{y}-C_2 \right)dy\\[1ex]
&\geq -C_1 \log{y_n}+C_3
\end{align*}
for some $C_3\in \mathbb{R}$. Taking the limit $y_n \searrow 0$ ($n\to \infty$), we obtain $\phi(z_n)\to \infty$.

Finally, we shall prove the statement (iii) under the assumption (ii). By the third component of \eqref{eq:n_max}, we have
\[
\frac{1}{\sqrt{1-|\nabla{\psi}(z)|^2}}=\frac{1+|\omega(w)|}{1-|\omega(w)|}
\]
for $z=f(w)$. Since $|\nabla{\psi}(z)|\to 1$, ($z\to I$), the analytic dilatation $\omega$ also satisfies $|\omega(w)|\to 1$ ($z\to I$). Moreover, the first and second  components of \eqref{eq:n_max} yield the relation
\[
\sqrt{\omega(w)}=\frac{1-|\omega(w)|}{\sqrt{1-|\nabla{\psi}(z)|^2}}\psi_z(z)=(1+|\omega(w)|)\psi_z(z).
\]
Therefore, we obtain $\omega(w)\to 1$, ($z\to I$) since $(\psi_x(z),\psi_y(z))\to (1,0)$, ($z\to I$).

Let $J'=C(f^{-1},I)=\cup_{z\in I}\ C(f^{-1},z)$, and let $E$ be the set of discontinuous points of $\hat{f}$ as in Lemma \ref{lemma:HS}. If we assume that $|J'|>0$, where $|J'|$ denotes the Lebesgue measure of $J'\subset \partial \mathbb{D}\cong \mathbb{R}/\mathbb{Z}$, then $|J'\setminus E|>0$ since $E$ is countable. 
By Lemma \ref{lemma:HS}, for each $e^{i\theta}\in J'\setminus E$, the unrestricted limit of $f(w)$ as $w\rightarrow e^{i\theta}$ exists and belongs to $I$. In particular, the radial limit $z=f(re^{i\theta})$ converges to a  point in $I$ as $r\to 1$, and hence we have $\omega(re^{i\theta}) \to 1$ as $r\to 1$. By F.-M.~Riesz's theorem (cf.~\cite[p.~220, Theorem A.3]{M}), we have $\omega \equiv 1$, which contradicts to the fact that $|\omega|<1$ in $\mathbb{D}$. Therefore, we conclude that $|J'|=0$.
On the other hand, if there does not exist $w_0\in E$ such that $I\subset C(f,w_0)$, then there are at least two distinct points $w_1, w_2\in J'$. Since $f$ is a homeomorphism, one of the two arcs connecting $w_1$ and $w_2$ on $\partial{\mathbb{D}}$ is included in $J'$. Then $|J'|>0$, which is a contradiction.
\hfill $\square$
%==========================END============================

%==========================CONVERSE of THM1============================
\subsection{Converse of Theorem \ref{thm:1}} \label{subsec:converseThm1}
Next, we discuss when the third condition (iii) in Theorem \ref{thm:1} conversely implies the conditions (i) and (ii) here. 

Let $\Omega,\ \phi,\ \psi$ and $f=h+\overline{g}$ be as in Theorem \ref{thm:1}.
We suppose that the boundary function $\hat{f}$ of $f$ admits a discontinuous point $w_0 \in E$, where $E$ denotes the set of discontinuous points of $\hat{f}$. As mentioned previously, in this case, the one-sided limits $z_0^{\pm}=\hat{f}(w_0^{\pm})$ exist in $\partial \Omega$ and are different. Such a discontinuous point is usually called a {\it jump point} (or a {\it discontinuous point of the first kind}).

\begin{definition}
		Let $\sigma \colon [0,\infty)\rightarrow [0,\infty)$ be a monotone increasing continuous function with $\sigma(0)=0$. We say that $w_0 \in E$ is {\it $\sigma$-regular} if there exist $\delta, C>0$ such that the following inequality holds whenever $|t|<\delta$ and $w_0 e^{it}\in \partial \mathbb{D} \setminus E$:
		\begin{align*}
			|\hat{f}(w_0 e^{it})-z_0^+|\leq C\sigma(|w_0 e^{it}-w_0|) &\ \  {\rm if}\ \ (t>0),\\
			|z_0^--\hat{f}(w_0 e^{it})|\leq C\sigma(|w_0-w_0 e^{it}|) &\ \ {\rm if}\ \ (t<0).
		\end{align*}
		Further, if we can take $\sigma(t)=t^{\lambda}$ for some $0<\lambda\leq 1$, then we say that $w_0$ is {\it $\lambda$-H\"older regular}.
\end{definition}

A function $\mu\colon\mathbb{D}\to \mathbb{C}$ is said to have a {\it non-tangential limit} (or {\it angular limit}) $a\in\mathbb{C}$ at $\zeta\in \partial \mathbb{D}$ if $\mu(w)$ converges to $a$ as $w\to \zeta$ in each Stolz angle $A_{\alpha}=A_{\alpha}(\zeta)=\{w\in \mathbb{D} \mid -\alpha< \arg(1-\overline{\zeta}w)<\alpha, |\zeta-w|<\cos \alpha\},\ 0<\alpha<\pi/2$. In this article, we denote this by
\[
		\alim_{w\to \zeta} \mu (w)=a.
\]
Then, we have the following statement.

\begin{theorem}\label{thm:5}
	Under the assumptions mentioned above, if $w_0$ is $\lambda$-H\"older regular, then there exists a constant $M\in \mathbb{C}$ such that the following holds$:$
	\[
		\alim_{w\rightarrow w_0} \left|F(w)-c\frac{|z_0^+-z_0^-|}{\pi}\log(w-w_0)-M\right|=0.
	\]
Here, $F=F_f$ is defined by $(\ref{eq:F})$, and $c=1$ or $-1$ which is determined by the choice of the branch of $\sqrt{h'g'}$.
\end{theorem}

\noindent
Recall that $\phi\circ f={\rm Re}(F)$ and $\psi\circ f={\rm Im}(F)$ hold up to additive constants.
Therefore, if $w_0$ is $\lambda$-H\"older regular, then there exist constants $M_1, M_2\in \mathbb{R}$
such that
\begin{align}
	&\alim_{w\rightarrow w_0} \left|\varphi\circ f(w)-c\frac{|z_0^+-z_0^-|}{\pi}\log|w-w_0|-M_1\right|=0,\label{eq:Rpart}\\[1ex]
	&\alim_{w\rightarrow w_0} \left|\psi\circ f(w)-c\frac{|z_0^+-z_0^-|}{\pi}\alpha(w;w_0)-M_2\right|=0. \label{eq:Ipart}
\end{align}
Here, we set $\alpha(w;w_0)=\arg(i(1-\overline{w_0}w))$. Combining (\ref{eq:Rpart}) and Theorem \ref{thm:1}, the following corollary is immediately obtained:

\begin{corollary} \label{cor:1}
		Under the same assumptions as in Theorem \ref{thm:5}, suppose $w_0$ is $\lambda$-H\"oder regular. Then $({\rm i}')$ $\phi$ tends to plus or minus infinity on $I_0$, and $({\rm ii}')$ $\psi$ tamely degenerates to a future or past-directed lightlike line segment on $I_0$, respectively. Here $I_0$ is the open segment given by removing the endpoints from the segment $C(f,w_0)$.
\end{corollary}

\noindent
The statement $({\rm ii}')$ follows from $({\rm i}')$ and Theorem \ref{thm:1}. However, (\ref{eq:Ipart}) and the well-known boundary behavior of $f$
\[
	\lim_{w\rightarrow w_0}\left| f(w)-\left\{z_0^+ \left(1-\frac{\alpha(w;w_0)}{\pi}\right)+z_0^- \frac{\alpha(w;w_0)}{\pi}\right\} \right| =0,
\]
actually implies that $\psi|_{I_0}$ parametrizes a future or past-directed lightlike line segment, see Figure \ref{zu15}.
\begin{figure}[htb]
 \begin{center}
   \begin{tabular}{c@{\hspace{0.3cm}}c@{\hspace{-0.4cm}}c}
        \includegraphics[scale=0.90]{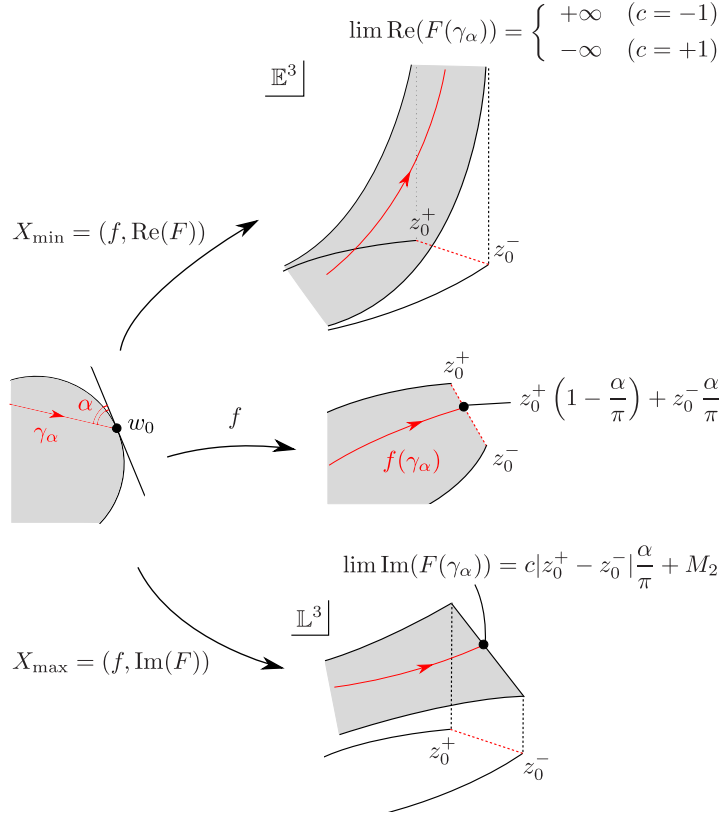} 
    \end{tabular}
\caption{Duality of boundary behavior of minimal and maximal surfaces around a discontinuous point $w_0$ of $f$.}
\label{zu15}
 \end{center}
\end{figure}

\begin{remark}
	Theorem \ref{thm:5} and  Corollary \ref{cor:1} can be extended to a wider class of discontinuous points, by appealing to the estimation method given by Shiga in \cite[Theorem 1]{S}. More precisely, let $\sigma \colon [0,\infty)\to [0,\infty)$ be a monotone increasing continuous function which satisfies $\sigma(0)=0$, and the following three conditions:
	\begin{itemize} \setlength{\itemsep}{1ex}
		\item doubling property: there exists $C>0$ such that $0<s\leq t\leq 2s$ implies $\sigma(s)\leq \sigma(t)\leq C\sigma(s)$,
		\item for each $0<\lambda\leq 1$, there exists $\delta>0$ such that $t^{\lambda}\leq \sigma(t)$ if $0<t<\delta$,
		\item Dini condition: $\displaystyle \int_0^1 \frac{\sigma(t)}{t}dt<+\infty$.
	\end{itemize}
\noindent
Then, for any $\sigma$-regular discontinuous point of $\hat{f}$, the same conclusions as in Theorem \ref{thm:5} and Corollary \ref{cor:1} hold.
\end{remark}

%==========================PROOF============================
\noindent
{\bf Outline of the proof of Theorem \ref{thm:5}}.
Since only standard and easy calculations are needed to prove Theorem \ref{thm:5}, we give only an outline.

We may assume $0<\lambda<1$. For $w_0=e^{i\theta_0}$, define $V\colon \partial \mathbb{D} \to \mathbb{C}$ by $V(e^{it})=z_0^+-z_0^-$ if $\theta_0<t<\theta_0+\pi$ and $V(e^{it})=0$ otherwise, and let $W=\hat{f}-V$. Then $W$ satisfies the $\lambda$-H\"older condition at $w_0$ in the usual sense. Recall that $f=h+\overline{g}$ can be written as the Poisson integral of $\hat{f}$. Thus we have
\begin{align*}
	h'(w)=\frac{1}{2\pi i}\int_{\partial \mathbb{D}} \frac{\hat{f}(\zeta)}{(\zeta-w)^2}d\zeta
			=\frac{1}{2\pi i}\int_{\partial \mathbb{D}} \frac{V(\zeta)}{(\zeta-w)^2}d\zeta+\frac{1}{2\pi i}\int_{\partial \mathbb{D}} \frac{W(\zeta)}{(\zeta-w)^2}d\zeta,\\[1ex]
	g'(w)=\frac{1}{2\pi i}\int_{\partial \mathbb{D}} \frac{\overline{\hat{f}(\zeta)}}{(\zeta-w)^2}d\zeta
			=\frac{1}{2\pi i}\int_{\partial \mathbb{D}} \frac{\overline{V(\zeta)}}{(\zeta-w)^2}d\zeta+\frac{1}{2\pi i}\int_{\partial \mathbb{D}} \frac{\overline{W(\zeta)}}{(\zeta-w)^2}d\zeta.
\end{align*}
Easy calculations show that
\begin{align*}
	h'(w)=-\frac{z_0^+-z_0^-}{2\pi i} \frac{1}{w-w_0}\left\{1+ O\left(|w-w_0|^{\lambda}\right) \right\},\\[1ex]
	g'(w)=-\frac{\overline{(z_0^+-z_0^-)}}{2\pi i} \frac{1}{w-w_0}\left\{1+ O\left(|w-w_0|^{\lambda}\right) \right\},
\end{align*}
as $w\to w_0$ on each Stolz angle at $w_0$. Therefore, we have
\[
	2i \sqrt{h'g'}=c\frac{|z_0^+-z_0^-|}{\pi} \frac{1}{w-w_0}\left\{1+ O\left(|w-w_0|^{\lambda}\right) \right\},
\]
as $w\to w_0$ on each Stolz angle. This implies the desired conclusion. \hfill $\square$
%==========================END of PROOF============================

%==========================POLYGONAL CASE============================
\subsection{Polygonal case}

In the case where $\Omega$ is a polygonal domain, some intensive studies are found in \cite{BW}, \cite{JS}, \cite{MS}, \cite{W}. In this case, by using the poisson integrals of step functions, we can give solutions of the following two boundary value problems for the minimal surface equation and the maximal surface equation, simultaneously. See also Section \ref{subsec:examples}.

\begin{corollary}\label{cor:polygonal bdr}
Let $\Omega\subset \mathbb{C}$ be a polygonal domain with open segment edges $I_j$ $(j=1,2,\dots, n)$, such that $I_j$ and $I_{j+1}$ have a common vertex $z_j$, where $I_{n+1}:=I_1$. 
Let $\phi\colon \Omega\rightarrow \mathbb{R}$ be a solution of the minimal surface equation, $\psi\colon \Omega\rightarrow \mathbb{R}$ its dual and $f=h+\overline{g}\colon\mathbb{D}\rightarrow \Omega$ the corresponding orientation-preserving univalent harmonic mapping. Then the following statements are equivalent.
\begin{itemize}
\item[(i)] $\phi(z)\to +\infty$ $(z\to I_j)$ or $\phi(z)\to -\infty$ $(z\to I_j)$ for each $j=1,2,\dots,n$.
\item[(ii)] $\psi$ tamely degenerates to a lightlike line segment on each $I_j$ for $j=1,2,\dots,n$.
\item[(iii)] The boundary function $\hat{f}$ is a step function on $\partial \mathbb{D}$ taking values in $\{z_j\mid j=1,2,\dotsc,n\}$.
\end{itemize}
\end{corollary}

\noindent
The proof of the equivalence (i) and (iii) can be found in \cite[Theorem 1]{BW} and \cite[Theorem 2]{W}. The following proof is almost the same way, however, we give a detailed proof since it gives an important observation for the later sections.

\begin{proof}
The equivalence (i) and (ii) follows from Theorem \ref{thm:1}. If we assume (ii), then there exists a discontinuous points $w_j$ $(j=1,2,\dotsc,n)$ of the boundary function $\hat{f}$ such that $I_j\subset C(f,w_j)$. When $w_j\neq w_{j+1}$, we can take an open arc $J_j\subset \partial{\mathbb{D}}$ joining $w_j$ and $w_{j+1}$ which does not contain the other $w_k$ $(k\neq j, j+1)$. Indeed, we can take two curves $\gamma_i\colon [0,1)\to \mathbb{D}$  $(i=1,2)$ which do not have intersections except the common starting point $\gamma_1(0)=\gamma_2(0)$ and satisfy
\[
\lim_{t\to1}{\gamma_1(t)}=w_j,\quad \lim_{t\to1}{\gamma_2(t)}=w_{j+1}\quad \text{and}\quad  \lim_{t\to1}{f(\gamma_1(t))}= \lim_{t\to1}{f(\gamma_2(t))}=z_j.
\]
Since $f$ is a homeomorphism from $\mathbb{D}$ to $\Omega$, the curve $\Gamma=f(\gamma_1)\cup f(\gamma_2)\cup\{z_j\}$ is a Jordan closed curve. Considering the bounded domain $D'\subset \Omega$ enclosed by $\Gamma$ and its preimage $D=f^{-1}{(D')}$, one can see that $J_j=\partial{D}\cap \partial{\mathbb{D}}$ does not contain any other $w_k$ except $w_j$ and $w_{j+1}$ since each $C(f,w_k)$ contains a line segment $I_k$ but $\partial{D'}\cap \partial{\Omega}=\{z_j\}$. 
This argument also shows that $\hat{f}|_{J_j}\equiv z_j$, which proves (iii). 

Finally, we assume the third condition (iii). Then it can be easily shown that $I_j\subset C(f,w_j)$ for some discontinuous point $w_j$ of $\hat{f}$ for each $j=1,\ldots,n$, by using the fact that $\partial\Omega=\bigcup _{w_k \in E}C(f,w_k)$ and $I_j\cap \{z_1,\ldots z_n\}=\emptyset$. Since $\hat{f}$ is a step function, each discontinuous point is $\lambda$-H\"older regular for $\lambda=1$. Thus Corollary \ref{cor:1} shows (i).
\end{proof}

\begin{remark}\label{remark:interior_angle}
In the proof of Corollary \ref{cor:polygonal bdr}, discontinuous points $w_j$ and $w_{j+1}$ may satisfy $w_j=w_{j+1}$. It should be emphasized that this can occur only when the interior angle of $I_j$ and $I_{j+1}$ is equal to $\pi$, since the cluster point set $C(f,w_j)$ is a line segment. Later, we will see that the condition for the interior angle strongly affects the boundary behavior of minimal and maximal surfaces. See Section \ref{subsec:reflection}.
\end{remark}

%==========================JENKINS SERRIN THM============================
Next, we recall the Jenkins-Serrin theorem in \cite{JS}. 
Let $\Omega\subset \mathbb{R}^2$ be a bounded simply connected Jordan domain whose boundary consists of a finite number of open line segments $A_1,\ldots,A_k,B_1,\ldots,B_l$ and a finite number of open convex arcs $C_1,\ldots,C_m$ together with their endpoints. For each of families $\{A_j\}, \{B_j\}$ and $\{C_j\}$, assume that no two of the elements meet to form a convex corner. Further, for a polygonal domain $P\subset \Omega$ whose vertices are in the set of the endpoints, let $\alpha=\alpha_P$ and $\beta=\beta_P$ denote respectively, the total length of $A_j$ such that $A_j\subset\partial P$ and the total length of $B_j$ such that $B_j\subset\partial P$, and let $\gamma=\gamma_P$ be the perimeter of $P$. Under this situation, we consider the following boundary value problem for the minimal surface equation$:$ for a prescribed piecewise continuous data $\widehat{\phi}_j:C_j\to \mathbb{R}$ on $C_j,\ j=1,\ldots,m$,
\begin{itemize}
	\item $\phi$ tends to plus infinity on $A_j,\ j=1,\ldots,k$,
	\item $\phi$ tends to minus infinity on $B_j,\ j=1,\ldots,l$,
	\item $\phi=\widehat{\phi}_j$ on $C_j,\ j=1,\ldots,m$.
\end{itemize}
\noindent
Then Jenkins and Serrin obtained the following theorem in \cite{JS}.
\begin{fact}\label{fact:JS}
If $\{C_j\}$ is non-empty, then the above boundary value problem for the minimal surface equation is solvable for arbitrary assigned data, if and only if, 
\begin{equation}\label{eq:JS_condition1}
		2\alpha<\gamma \ \ \ \text{and}\ \ \  2\beta <\gamma
\end{equation}
hold for each polygonal domain taken as above. The solution is unique if it exists.

If $\{C_j\}$ is empty, then there exists a solution, if and only if,
\begin{equation}\label{eq:JS_condition2}
		\alpha_{\Omega}=\beta_{\Omega}
\end{equation}
holds and $(\ref{eq:JS_condition1})$ hold for each polygonal proper subdomain taken as above. The solution is unique up to an additive constant if it exists.
\end{fact}

Here, we restrict ourselves to the case where $\{C_j\}$ is empty. By Corollary \ref{cor:polygonal bdr}, the existence conditions \eqref{eq:JS_condition1} and \eqref{eq:JS_condition2} for minimal graphs are translated to the conditions for maximal surfaces, as follows: 
 Suppose that $\{C_j\}$ is empty, and there exists a solution $\phi$ of the above infinite boundary value problem. Let $\psi$ be the dual solution. Then, the boundary of $\Sigma_{\max}={\rm graph}(\psi)$ consists of future-directed lightlike line segments on each $A_j$ and past-directed lightlike line segments on each $B_j$, by Corollary \ref{cor:polygonal bdr}.
Therefore, the latter condition $\alpha_{\Omega}=\beta_{\Omega}$ just means that $\Gamma=\partial \Sigma_{\max}$ is a closed curve.
On the other hand, the former condition \eqref{eq:JS_condition1} corresponds to the statement that each line segment $l$ connecting two vertices of $\Gamma$ is spacelike whenever $\Pi(l) \subset \Omega$, where $\Pi$ is the projection from $\mathbb{L}^3$ to the $xy$-plane. 
In fact, if we consider a polygon $P\subset \Omega$ which is one of the connected components of $\Omega\setminus \Pi(l)$, then the former condition $2\alpha<\gamma$ and $2\beta<\gamma$ is equivalent to $|\alpha-\beta|<|\Pi(l)|=\gamma-\alpha-\beta$. This means that $l\subset \mathbb{L}^3$ is spacelike since $|\alpha-\beta|$ is exactly the difference between the heights of endpoints of $l$.

Therefore, in conjunction with Corollary \ref{cor:polygonal bdr} one can solve the following boundary value problem for maximal surfaces as a counter part of Jenkins-Serrin's result.

\begin{corollary}\label{cor:JS}
 Let $\Omega\subset \mathbb{R}^2$ be a polygonal domain whose boundary consists of $A_j, B_j\subset \partial{\Omega}$ as in Fact $\ref{fact:JS}$. Let $\Gamma \subset \mathbb{L}^3$ be a polygonal curve which consists of future-directed lightlike line segments on $A_j$ and past-directed lightlike line segments on $B_j$. Then, there exists a solution of the maximal surface equation over $\Omega$ which tamely degenerates to each edge of $\Gamma$ if and only if $\Gamma$ is a closed curve and each line segment $l$ connecting vertices of $\Gamma$ is spacelike whenever the projection of $l$ to the $xy$-plane lies on  $\Omega$. Moreover, such solution is unique and can be written by the Poisson integral of some step function taking values on the vertices of $\Omega$.
\end{corollary}

Here, we remark that Bartnik and Simon \cite{BS} discussed the existence and uniqueness of solutions of boundary value problems with prescribed boundary values and mean curvatures for weakly spacelike hypersurfaces in the Lorentz-Minkowski space. Their results also include the existence and uniqueness of solutions of the area maximizing problems for weakly spacelike graphs in the sense of \cite[(1.3)]{BS}.

%============================================================SECTION 4========
\section{Extension via reflection principle and symmetry} \label{sec:4} 
In this section, we investigate more details of boundary behavior of minimal and maximal surfaces discussed in the previous section, and also study their conjugate surfaces and symmetry relations.

%==========================REFLECTION PRINCIPLE============================
\subsection{Interior angle of boundary edges and reflection principle} \label{subsec:reflection}
By the reflection principle for harmonic mappings, we extend surfaces across vertical lines on minimal surfaces and shrinking singularities on maximal surfaces.

As in the previous sections, we let $\Omega$ be a bounded simply connected Jordan domain, $\phi$ a solution of the minimal surface equation over $\Omega$, $\psi$ the dual of $\phi$, and $f\colon \mathbb{D}\to \Omega$ the corresponding orientation-preserving univalent harmonic mapping. We consider the case where $\partial \Omega$ contains adjacent segments $I_1$ and $I_2$ having a common endpoint $z_0$ with interior angle $\alpha$, and $\phi$ tends to plus or minus infinity on each of $I_1$ and $I_2$. Then, Theorem \ref{thm:1} shows that there exist discontinuous points $w_1, w_2 \in \partial \mathbb{D}$ of $\hat{f}$ such that $I_j \subset C(f,w_j),\ j=1,2$. As mentioned in Remark \ref{remark:interior_angle}, it may occur that $w_1=w_2$ only when $\alpha=\pi$.

\begin{remark}
	We emphasize that the only reason for assuming that $\Omega$ is a bounded simply connected Jordan domain is to apply Theorem \ref{thm:1} or Hengartner-Schober's result (Lemma \ref{lemma:HS}) to $f$. Therefore, almost every discussion in this section can be applied to the case where only the simply connectedness is assumed for $\Omega$, by restricting the arguments to an appropriate subdomain. However, the simply connectedness is needed for the existence of the dual.
\end{remark}

We first see relations between the interior angle $\alpha$ and the signs of $\phi$ over $I_1$ and $I_2$. So far, the sign changing of $\phi$ on boundary edges were discussed by Bshouty and Weitsman in \cite{BW}, and their function theoretical approach is based on the works in \cite{BH}, \cite{HS2}.

\begin{proposition} \label{prop:signAlter}
	Assume $\phi$ tends to plus or minus infinity on each of adjacent segments $I_1$ and $I_2$ having a common endpoint $z_0$ with interior angle $\alpha<\pi$. Then the signs of $\phi$ on $I_1$ and $I_2$ are different.
\end{proposition}

This proposition is already proved in \cite{JS} and \cite{W} in different ways; the former used the {\it straight line lemma} in \cite[Section 4]{JS}, and the latter used the generalized maximum principle, respectively. By applying Theorem \ref{thm:1}, we have an another simple spacetime geometrical proof as follows:

\begin{proof}
	To obtain a contradiction, we assume $\phi$ tends to plus infinity on $I_1$ and $I_2$. Then Theorem \ref{thm:1} implies that its dual $\psi$ tamely degenerates to future-directed lightlike line segments on $I_1$ and $I_2$. We can find $z_1\in I_1$ and $z_2 \in I_2$ such that the straight line segment $l$ joining $z_1$ and $z_2$ lies in $\Omega$ since $\alpha<\pi$. The triangle inequality shows
	\begin{equation*}
		|z_2-z_1|<|z_2-z_0|+|z_0-z_1|=|\psi(z_2)-\psi(z_1)|. \label{eq:nonLips}
	\end{equation*} 
However, $\psi$ is $1$-Lipschitz on any convex set of $\Omega$, in particular on $l$, since $|\nabla \psi|<1$. We have a contradiction.
\end{proof}

When $\alpha>\pi$, the sign of $\phi$ may change in general. Such an example was constructed in \cite{MS} (see the left of Figure \ref{fig:branch}). In the case where $\alpha=\pi$ and the signs are the same, the following ``{\it removable singularity theorem}'' does hold. 

\begin{proposition} \label{prop:removablity}
	Assume $\phi$ tends to plus infinity on adjacent segments $I_1$ and $I_2$ having a common endpoint $z_0$ with interior angle $\alpha=\pi$. Then $\phi$ tends to plus infinity on $I=I_1\cup \{z_0\} \cup I_2$.
\end{proposition}  

\begin{proof}
	Take a small disk $D$ centered at $z_0$ so that $D\cap (\partial \Omega\setminus I)=\emptyset$. Then, the Jenkins-Serrin theorem (Fact \ref{fact:JS}) shows that there exists a solution $\widetilde{\phi}$ of the minimal surface equation over $D\cap \Omega$ such that $\widetilde{\phi} \to +\infty$ on $D\cap I$ and $\widetilde{\phi}=\phi$ on $C_1=\partial D\cap \Omega$. On the other hand, if we consider the boundary value problem for $A_1=I_1\cap D$, $A_2=I_2\cap D$, and $C_1$ with the prescribed boundary data $\widehat{\phi}_1=\phi|_{C_1}$ in the setting of the Jenkins-Serrin theorem, then we already have two solutions $\phi$ and $\widetilde{\phi}$. Thus, the uniqueness shows $\phi=\widetilde{\phi}$ on $D\cap \Omega$. Since $\widetilde{\phi}(z)\to +\infty$ as $z\to z_0$, we have the conclusion.
\end{proof}
Also, the following lemma on discontinuous points of $\hat{f}$ holds.
\begin{lemma} \label{prop:discontiPoints}
	Assume $\phi$ tends to plus or minus infinity on each of adjacent segments $I_1$ and $I_2$ having a common endpoint $z_0$ with interior angle $\alpha$. Let $w_1, w_2\in \partial \mathbb{D}$ be corresponding discontinuous points of $\hat{f}$ to $I_1$ and $I_2$, respectively. Then, $w_1=w_2$ if and only if $\alpha=\pi$ and the signs of $\phi$ on $I_1$ and $I_2$ are the same. Further, if $w_1\neq w_2$, then $\hat{f}\equiv z_0$ on one of the arcs in $\partial \mathbb{D}$ joining $w_1$ and $w_2$.
\end{lemma}

\begin{proof}
First, the latter statement is proved in the same way as in the proof of Corollary \ref{cor:polygonal bdr}. Next, the sufficiency of the condition $w_1=w_2$ in the former statement follows immediately from Proposition \ref{prop:removablity}. Thus, we prove the necessity.

Assume $w_1=w_2$. Then $\alpha=\pi$ as stated in Remark \ref{remark:interior_angle}. Thus, to obtain a contradiction, we suppose now that the signs of $\phi$ on $I_1$ and $I_2$ are different. Let $\psi$ be the dual of $\phi$. By taking an appropriate subdomain of $\mathbb{D}$ whose boundary contains $w_1$ and applying the uniformization theorem if necessary, we may assume that the one-sided limits $z_1^{\pm}:=\hat{f}(w_1^{\pm})$ are the endpoints of $I_1$ and $I_2$ different from $z_0$, and that $t:=\psi\circ f$ is continuous on $\partial \mathbb{D}\setminus \{w_1\}$ and $t(w_1^{\pm})=\psi(z_1^{\pm})$. Since $t$ can be written as the Poisson integral of its boundary function, the well-known argument shows $C(t,w_1)=[a,b]$ where $a=\min\{\psi(z_1^+),\psi(z_1^-)\}$ and $b=\max\{\psi(z_1^+),\psi(z_1^-)\}$ (see, Section \ref{subsec:converseThm1} or Section \ref{subsec:sym}). This implies, in particular, $a<\psi(z_0)<b$. On the other hand, $\psi$ tamely degenerates to lightlike segments on $I_1$ and $I_2$ with different causal directions, by Theorem \ref{thm:1}. 
This implies that $\psi(z_0)<a$ or $b<\psi(z_0)$. We have a contradiction.
\end{proof}

Under the above observations, we have the following reflection principle:

\begin{theorem} \label{thm:reflection}
	Let $\Omega$ be a bounded simply connected Jordan domain, whose boundary contains segments $I_1$ and $I_2$ having a common endpoint $z_0$ with interior angle $\alpha$. And let $\phi$ be a solution of the minimal surface equation over $\Omega$, and $\psi$ its dual.
	
	Assume that $\varphi$ tends to plus or minus infinity on $I_1$ and $I_2$, respectively, and that the signs of $\phi$ on $I_1$ and $I_2$ differ if $\alpha=\pi$. Then the following statements hold$:$
	
	\begin{itemize}   \setlength{\leftskip}{-3ex}
			\item[(a)] The graph of $\phi$ is extended to a generalized minimal surface $X_{\mathrm{min}}\colon \mathbb{D}\rightarrow \mathbb{E}^3$ such that $X_{\mathrm{min}}(\mathbb{D}\cap \mathbb{H})=\mathrm{graph}(\phi)$ and $X_{\mathrm{min}}$ admits a vertical line $L=X_{\mathrm{min}}(\mathbb{D}\cap\mathbb{R})$ on $z_0$ with the symmetry $X_{\mathrm{min}}(\overline{w})=\sigma\circ X_{\mathrm{min}}(w)$, where $\sigma$ is the $\pi$-rotation with respect to $L$. 
			\item[(b)] The graph of $\psi$ is extended to a generalized maximal surface $X_{\mathrm{max}}\colon \mathbb{D} \rightarrow \mathbb{L}^3$ such that $X_{\mathrm{max}}(\mathbb{D}\cap \mathbb{H})=\mathrm{graph}(\psi)$ and $X_{\mathrm{max}}$ has a shrinking singularity $p_0=(z_0,\psi(z_0))=X_{\mathrm{max}}(\mathbb{D}\cap \mathbb{R})$ with the symmetry $X_{\mathrm{max}}(\overline{w})=\tau\circ X_{\mathrm{max}}(w)$, where $\tau$ is the point symmetry with respect to $p_0$.
	\end{itemize}
\end{theorem}

\begin{proof}
The statement (a) seems to be well-known at least when $\Omega$ is convex (see \cite{Karcher} and \cite[p.~218]{DHS}, for example), but for the convenience of readers and the purpose of clarity, we give a proof which includes this case. 
Since $|\nabla\psi|<1$, the function $\psi$ extends continuously to $I_1\cup\{z_0\}\cup I_2$. Thus, we may assume that $z_0=0$ and $\psi(z_0)=0$. Hence, $\sigma$ and $\tau$ are written as $\sigma(x,y,t)=(-x,-y,t)$ and $\tau(x,y,t)=(-x,-y,-t)$, respectively. Let $f=h+\overline{g}\colon \mathbb{D}\rightarrow \Omega$ be the corresponding orientation-preserving univalent harmonic mapping.
By Theorem \ref{thm:1}, there are discontinuous points $w_1,w_2\in \partial{\mathbb{D}}$ of the the boundary function $\hat{f}\colon \partial{\mathbb{D}}\rightarrow \partial{\Omega}$ such that $I_j\subset C(f,w_j)$ $(j=1,2)$. Then, $w_1\neq w_2$ and $\hat{f}\equiv 0$ on an arc $J_0\subset \partial \mathbb{D}$ connecting $w_1$ and $w_2$, by the assumption and Lemma \ref{prop:discontiPoints}.
We can easily construct  a bi-holomorphic function $\Psi\colon \mathbb{D}\cap\mathbb{H} \rightarrow \mathbb{D}$ such that $\Psi((-1,1))=J_0$. Put 
\[
F(w)=2i\int_0^w\sqrt{h'(\zeta)g'(\zeta)}d{\zeta}+c,
\]
 where the constant $c\in \mathbb{C}$ is determined so that $\phi=\text{Re}(F)\circ f^{-1}$ and $\psi=\text{Im}(F)\circ f^{-1}$. Then for each $w\in (-1,1)$, we have
\[
f\circ \Psi(w)=z_0=0,\quad \text{Im}(F)\circ \Psi(w)=\psi(z_0)=0.
\]
By the reflection principle, we can extend $f\circ \Psi$ and $F\circ \Psi$ to a harmonic mapping and a holomorphic function on $\mathbb{D}$ satisfying $f\circ \Psi (\overline{w})=-f\circ \Psi(w)$ and $F\circ \Psi (\overline{w})=\overline{F\circ \Psi (w)}$, respectively. Using these extended maps, we obtain a generalized minimal surface $X_{\text{min}}=(f\circ \Psi, \text{Re}{(F)}\circ \Psi)\colon\mathbb{D}\rightarrow \mathbb{E}^3$ and a generalized maximal surface $X_{\text{max}}=(f\circ \Psi, \text{Im}{(F)}\circ \Psi)\colon\mathbb{D}\rightarrow \mathbb{L}^3$ satisfying 
\begin{align*}
X_{\text{min}}(\overline{w})&=(-f\circ \Psi (w), \text{Re}{(F)}\circ \Psi (w))=\sigma(X_{\text{min}}(w)),\\
X_{\text{max}}(\overline{w})&=(-f\circ \Psi (w), -\text{Im}{(F)}\circ \Psi (w))=\tau(X_{\text{min}}(w)),
\end{align*}
which are desired relations.
\end{proof}
	
\begin{remark} \label{remark:extended}
	Notice that the extended surfaces $X_{\min}=(f\circ \Psi, \text{Re}{(F)}\circ \Psi)$ and $X_{\max}=(f\circ \Psi, \text{Im}{(F)}\circ \Psi)$ also have representations $X_{\min}=(\widetilde{f},{\rm Re}(\widetilde{F}))$ and $X_{\max}=(\widetilde{f},{\rm Im}(\widetilde{F}))$, where $\widetilde{f}$ is a harmonic function and $\widetilde{F}=F_{\widetilde{f}}$ is a holomophic function defined by \eqref{eq:F}. However, it should be remarked that $\widetilde{f}$ is no longer univalent since $\widetilde{f}(w)=z_0$ on $(-1,1)$, and is orientation-reversing on the lower half part of the unit disk since $\widetilde{f}(\overline{w})=-\widetilde{f}(w)$. In this case, on the lower half part of the unit disk, $X_{\max}$ actually parametrizes the dual maximal graph of the minimal graph parametrized by $(\widetilde{f}, -{\rm Re}(\widetilde{F}))$, see Proposition \ref{prop:parametric_duality}.
\end{remark}

\begin{remark} \label{remark:branchpoint}
The generalized minimal surface $X_{\mathrm{min}}=(f\circ \Psi, \text{Re}{(F)}\circ \Psi)$ in Theorem \ref{thm:reflection} may have branch points on the vertical line segment $L=X_{\mathrm{min}}((-1,1))$.  Indeed, by the construction of $X_{\min}$, an easy calculation shows that each branch point on $L$ corresponds to a zero point of $h'$, the derivative of the holomorphic part of the corresponding harmonic mapping. Further, it can be easily seen that the monotonicity of $\text{Re}{(F)}\circ \Psi$ on $(-1,1)$ changes exactly at zero points of $h'$ of odd order. See the left-hand side of Figure \ref{fig:branch} and also \cite[Example 1]{MS}.
\end{remark}

\begin{figure}[htb]
 \begin{center}
   \begin{tabular}{c@{\hspace{0.3cm}}c@{\hspace{-0.4cm}}c}
        \includegraphics[height=5cm]{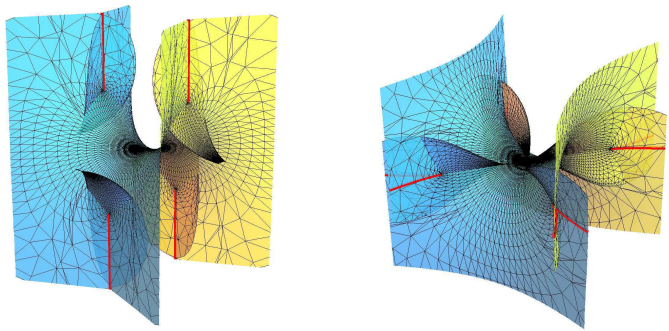} & &%{zu23.eps} &
    \end{tabular}
\caption{A minimal surface in Theorem \ref{thm:reflection} with branch points on vertical half lines and its conjugate minimal surface.}
\label{fig:branch}
 \end{center}
\end{figure}

%==========================CONJUGATE SURFACE============================
\subsection{Conjugate surface and symmetry, quantitative relation} \label{subsec:sym}
As proved by Karcher \cite{Karcher}, it is known that the Jenkins-Serrin minimal graphs in \cite{
JS} over equilateral convex $2k$-gons which diverge to plus or minus infinity alternately on each edge, have the conjugate graphs bounded by horizontal geodesics lying alternately in a top and a bottom symmetry planes. So by repeating reflections in these planes, the conjugate minimal graphs are extended to complete embedded minimal surfaces with vertical translation period, which are now referred as {\it saddle towers}. In this and the next subsections, we discuss symmetry relations under the conjugation and the dual operation in \eqref{eq:duality}.

We denote by $\xi^{\ast}$ the conjugate harmonic function of a real-valued harmonic function $\xi$.
Similarly, for a complex-valued harmonic function $\zeta=\xi+i\eta$, we denote $\zeta^{\ast}=\xi^{\ast}+i\eta^{\ast}$. 
Let $\phi$ be a solution of the minimal surface equation over $\Omega$, $\psi$ its dual, and $f=x+iy=h+\overline{g}\colon \mathbb{D}\to \Omega$ the corresponding harmonic mapping. Then, recall that $X_{\min}=(f,\phi\circ f)$ and $X_{\max}=(f, \psi\circ f)$ give isothermal parametrizations of $\Sigma_{\min}={\rm graph}(\phi)$ and $\Sigma_{\max}={\rm graph}(\psi)$, respectively. Further, we can write $\phi\circ f={\rm Re}(F)$ and $\psi \circ f={\rm Im}(F)$ for a holomorphic function
\begin{align*}
	F(w)=2i\int_0^w \sqrt{h'g'}d\zeta +c,
\end{align*}
for some constant $c\in \mathbb{C}$. Therefore, if we denote $t=\phi\circ f$, then $t^{\ast}=\psi \circ f$, that is, $X_{\min}=(x,y,t)=(f,t)$ and $X_{\max}=(x,y,t^{\ast})=(f,t^{\ast})$. On the other hand, 
\begin{align*}
	X_{\min}^{\ast}&=(x^{\ast},y^{\ast},t^{\ast})=(f^{\ast},t^{\ast})\colon \mathbb{D}\to \mathbb{E}^3, \\ 	
	X_{\max}^{\ast}&=(x^{\ast},y^{\ast},-t)=(f^{\ast},-t)\colon \mathbb{D}\to \mathbb{L}^3,
\end{align*}
are called the conjugate minimal surface of $X_{\min}$ and the conjugate maximal surface of $X_{\max}$, which are isometric to $X_{\min}$ and $X_{\max}$, respectively. We notice that each of the surfaces $X_{\min}$, $X_{\max}$, $X_{\min}^{\ast}$ and $X_{\max}^{\ast}$ is given by a combination of functions $f$, $f^{\ast}$, $t$, and $t^{\ast}$. Further, the following commutative diagram between the conjugation and the dual operation holds:
\[
  \begin{diagram}
    \node{X_{\text{min}}=(f,t)} \arrow{e,t}{\text{dual}} \arrow{s,l}{\text{conjugate}} \node{X_{\text{max}}=(f,t^*)} \arrow{s,r}{\text{conjugate}} \\
    \node{X^*_{\text{min}}=(f^*,t^*)} \arrow{e,t}{\text{dual}} \node{X^*_{\text{max}}=(f^*,-t)}
  \end{diagram}
\]
Therefore, to observe the boundary behavior of these four surfaces, it suffices to examine the boundary behavior of the only four functions $f$, $f^{\ast}$, $t$, and $t^{\ast}$.

\begin{remark}
Although the conjugate surfaces $X_{\min}^{\ast}$ and $X_{\max}^{\ast}$ are no longer graphs in general, by Proposition \ref{prop:parametric_duality}, we can define the duals of them as in the above diagram.
\end{remark}

\begin{remark}
On the above commutativity, we remark that the dual operation does not preserve ambient isometries in $\mathbb{E}^3$ and $\mathbb{L}^3$ as discussed by Ara\'{u}jo-Leite \cite{AL}.
\end{remark}

Assume now that $\Omega$ is a bounded simply connected Jordan domain, whose boundary contains three segments $I_1$, $I_2$, and $I_3$ which lie in the positive direction in this order, and that $I_j$ and $I_{j+1}$ have a common endpoint $z_j$ with interior angle $\alpha_j$ for $j=1,2$. Further, we suppose that $\phi$ tends to plus infinity on $I_2$, and tends to plus or minus infinity on $I_1$ and $I_3$ so that the signs on $I_j$ and $I_{j+1}$ are different if $\alpha_j =\pi$ for $j=1,2$. Then, Theorem \ref{thm:1} shows that there exist three discontinuous points $w_1, w_2, w_3\in \partial \mathbb{D}$ of $\hat{f}$ such that $I_j\subset C(f,w_j),\ j=1,2,3$. These three points are distinct, and it holds that $\hat{f}\equiv z_j$ on an arc $J_j\subset \partial \mathbb{D}$ which joins $w_j$ and $w_{j+1}$ for each $j=1,2$, by Proposition \ref{prop:discontiPoints}. Further, it is clear that $\overline{I_2}=C(f,w_2)$, and $w_1, w_2, w_3$ lie in $\partial \mathbb{D}$ in the counterclockwise direction in this order, since $f$ is an orientation-preserving homeomorphism. Under these settings, we consider the boundary behavior of the functions $f$, $f^{\ast}$, $t$, and $t^{\ast}$ near $w_2$, respectively.

First, we investigate the boundary behavior of $f$ and $f^{\ast}$. Recall that $f$ can be written as the Poisson integral of $\hat{f}$. Thus, the well-known argument for the harmonic measure (or a direct computation) implies,
\begin{align*}
	\begin{gathered}
	f(w)=\frac{z_1}{\pi}\left\{ \arg\left(\frac{w-w_2}{w-w_1}\right)-\frac{\arg(w_2-w_1)}{2}\right\}+\frac{z_2}{\pi}\left\{ \arg\left(\frac{w-w_3}{w-w_2}\right)-\frac{\arg(w_3-w_2)}{2}\right\}\\[1ex]
	+\frac{1}{2\pi}\int_0^{2\pi}{\rm Re}\left(\frac{e^{it}+w}{e^{it}-w}\right) W(e^{it})dt,
	\end{gathered}
\end{align*}
where $W=0$ on $J=J_1\cup \{w_2\}\cup J_2$, and $W=\hat{f}$ otherwise. Observe that $(\arg w)^{\ast}=-\log|w|$ and that the conjugate harmonic function of the third term on the right-hand side, which can be written as the conjugate Poisson integral
\begin{equation*}
	\frac{1}{2\pi}\int_0^{2\pi}{\rm Im}\left(\frac{e^{it}+w}{e^{it}-w}\right) W(e^{it})dt,
\end{equation*}
clearly tends to a constant as $w\to w_2$ since $W=0$ on $J$. We have
\begin{align}
	&\lim_{w\rightarrow w_2}\left| f(w)-\left\{z_2 \left(1-\frac{\alpha(w;w_2)}{\pi}\right)+z_1 \frac{\alpha(w;w_2)}{\pi}\right\} \right| =0,\label{eq:bdryf}\\[1ex]
	&\lim_{w\rightarrow w_2}\left| f^{\ast}(w)-\frac{z_2-z_1}{\pi}\log|w-w_2|-c_1 \right| =0, \label{eq:bdryfast}
\end{align}
for some constant $c_1\in \mathbb{C}$ and $\alpha(w;w_2)=\arg(i(1-\overline{w_2}w))$. It is shown that $f$ moves monotonically from $z_2$ to $z_1$ on $I_2$ when $\alpha(w;w_2)$ moves from $0$ to $\pi$ by \eqref{eq:bdryf}, and $f^{\ast}(w)$ diverges to infinity in $(z_1-z_2)$-direction as $w\to w_2$ by \eqref{eq:bdryfast}.

On the other hand, $t^{\ast}$ can be written as the Poisson integral of some bounded function, since $t^{\ast}$ is a bounded harmonic function, see \cite[p.72, Lemma 1.2]{Katz}. Further, $t^{\ast}=\psi\circ f$ satisfies $t^{\ast}(w)=\psi(z_1)$ on $J_1$ and $t^{\ast}(w)=\psi(z_2)$ on $J_2$. Notice that $\psi(z_2)-\psi(z_1)=|z_2-z_1|$, since $\psi$ tamely degenerates to a future-directed lightlike line segment on $I_2$ by Theorem \ref{thm:1}. Similarly to $f$ and $f^{\ast}$, we obtain
\begin{align}
	 &\lim_{w\to w_2}\left|\ t^{\ast}(w)+\frac{|z_2-z_1|}{\pi}\alpha(w;w_2)-\psi(z_2)\right|=0, \label{eq:bdrytast}\\[1ex]
	 &\lim_{w\to w_2}\left|\ t(w)+\frac{|z_2-z_1|}{\pi}\log|w-w_2|-c_2 \right| =0, \label{eq:bdryt}
\end{align}
for some constant $c_2\in \mathbb{R}$. Combining \eqref{eq:bdryf} to \eqref{eq:bdryt}, we immediately obtain the following relations between $X_{\min}$, $X_{\max}$, $X_{\min}^{\ast}$ and $X_{\max}^{\ast}$ together with quantitative relations at infinity.
\begin{theorem}
		Assume that $\partial \Omega$ contains three line segments $I_1,I_2$ and $I_3$ which lie in the positive direction in this order, and that $I_j$ and $I_{j+1}$ have a common endpoint $z_j$ with interior angle $\alpha_j,\ j=1,2$. Further, we suppose that $\phi$ tends to plus infinity on $I_2$, and tends to plus or minus infinity on $I_1$ and $I_3$ so that the signs on $I_j$ and $I_{j+1}$ are different if $\alpha_j =\pi$ for $j=1,2$. Then the following statements hold$:$
	\begin{itemize}   \setlength{\leftskip}{-3ex}
		\item $X_{\min}^{\ast}$ diverges to a vertical line segment of length $|I_2|=|z_2-z_1|$ at infinity in $(z_1-z_2)$-direction. 
		\item $X_{\max}^{\ast}$ diverges to $-\infty$ at infinity in $(z_1-z_2)$-direction. 
	\end{itemize}
Further, these vertical segment at infinity and infinite  point at infinity corresponds to the horizontal segment at infinity over $I_2$ which is a boundary of $X_{\min}$ as well as the future-directed lightlike line segment over $I_2$ which is a boundary of $X_{\max}$, under the conjugation and the dual operation.
\end{theorem}
\noindent
Here, ``$X_{\min}^{\ast}$ diverges to a vertical line segment of length $|I_2|$ at infinity in $(z_1-z_2)$-direction'' means that $f^{\ast}$ diverges to infinity in $(z_1-z_2)$-direction and $t^{\ast}$ moves on an interval of length $|I_2|$ as cluster points (see Figure \ref{zu18}).

\begin{figure}[htbp]
       \begin{center}
           \includegraphics[scale=.84]{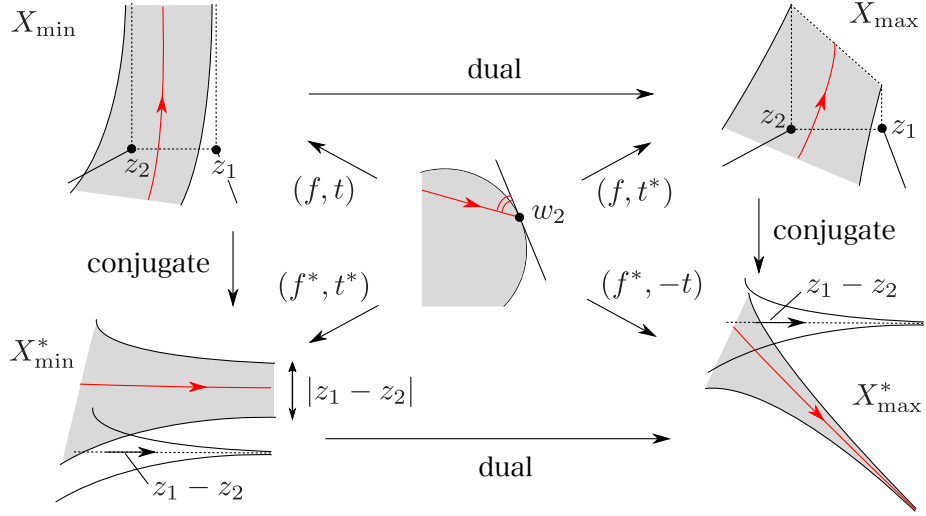} 
       \end{center}
       \caption{Boundary behavior of the four surfaces $X_{\min}$, $X_{\max}$, $X_{\min}^{\ast}$ and $X_{\max}^{\ast}$ around a discontinuous point $w_2$ of $f$.}   \label{zu18}
\end{figure} %\vspace{-2ex}

%==========================CONJUGATE SURFACE============================
\subsection{Singularity and reflection symmetry} \label{subsec:singularity}

We consider the following situation again. Let $\Omega$ be a bounded simply connected Jordan domain whose boundary contains adjacent segments $I_1$ and $I_2$ having a common endpoint $z_0=0$ with interior angle $\alpha$. Assume a solution $\phi$ of the minimal surface equation over $\Omega$ tends to plus or minus infinity on $I_1$ and $I_2$, so that the signs are different if $\alpha=\pi$. We let $\psi$ be the dual of $\phi$ with the normalization $\psi(0)=0$. Then recall that, in Section \ref{subsec:reflection}, we constructed a generalized minimal surface $X_{\min}\colon\mathbb{D}\to \mathbb{E}^3$ and a generalized maximal surface $X_{\max}\colon \mathbb{D}\to \mathbb{L}^3$ which parametrize ${\rm graph}(\phi)$ and ${\rm graph}(\psi)$ on $\mathbb{D}\cap \mathbb{H}$, and admit a vertical line segment $L$ over $z_0$ and shrinking singularity $p=(z_0,\psi(z_0))=(0,0)$ on the interval $(-1,1)$, respectively. Further, if we denote $X_{\min}=(f,t)$ with the corresponding harmonic function $f$ (see Remark \ref{remark:extended}), then they have the reflection symmetries
\begin{align*}
	&X_{\min}(\overline{w})=(f(\overline{w}),t(\overline{w}))=(-f(w),t(w))=\sigma\circ X_{\min}(w),\\
	&X_{\max}(\overline{w})=(f(\overline{w}),t^{\ast}(\overline{w}))=(-f(w),-t^{\ast}(w))=\tau\circ X_{\max}(w),
\end{align*}
where $\sigma$ and  $\tau$ denote the $\pi$-rotation with respect to $L$ and the point symmetry with respect to $p$, respectively. The conjugate surfaces are also defined for generalized surfaces in the same way, that is, $X_{\min}^{\ast}=(f^{\ast},t^{\ast})$ and $X_{\max}^{\ast}=(f^{\ast},-t)$. Since $t^{\ast}(w)=\psi(z_0)=0$ for $w\in (-1,1)$, the curve $\Gamma=X_{\min}^{\ast}((-1,1))$ is contained in the $xy$-plane. Notice that, since $f(w)=0$ on $(-1,1)$, the reflection principle for harmonic functions implies that $f^{\ast}(\overline{w})=f^{\ast}(w)$ holds. Thus, they also have the following reflection symmetries,
\begin{align*}
	&X_{\min}^{\ast}(\overline{w})=(f^{\ast}(\overline{w}),t^{\ast}(\overline{w}))=(f^{\ast}(w),-t^{\ast}(w))=\rho \circ X_{\min}^{\ast}(w),\\
	&X_{\max}^{\ast}(\overline{w})=(f^{\ast}(\overline{w}),-t(\overline{w}))=(f^{\ast}(w),-t(w))=X_{\max}^{\ast}(w),
\end{align*}
where $\rho$ denotes the planar symmetry with respect to the $xy$-plane. 

On the other hand, there are several know facts related to the above situation as follows: A straight-line on $X_{\text{min}}$ corresponds to a planar geodesic on $X^*_{\text{min}}$, which is also a curvature line, see \cite[Section 3.4]{DHS} for example. Thus, in particular, the vertical line segment $L=X_{\min}((-1,1))$ corresponds to the horizontal geodesic curvature line $\Gamma=X_{\min}^{\ast}((-1,1))$. A shrinking singularity on $X_{\text{max}}$ also corresponds to a folding singularity on $X^*_{\text{max}}$ as proved in \cite[Lemma 4.2, Theorem 4.3]{KY}, see also \cite[Proposition 2.14]{Okayama}. Therefore, we also conclude that, in particular, the shrinking singularity $p=(0,0)=X_{\max}((-1,1))$ corresponds to the null curve $C=X_{\max}^{\ast}((-1,1))$ as the image of the set of fold singular points $(-1,1)$. 

The following symmetry assertions summarize the above discussions, see also Figure \ref{zu19}.
\begin{theorem} \label{thm:reflectionSymmetries}
		Let $\Omega$ be a bounded simply connected Jordan domain whose boundary contains adjacent segments $I_1$ and $I_2$ having a common endpoint $z_0$ with interior angle $\alpha$. Further, we let $\phi$ be a solution of the minimal surface equation over $\Omega$, and $\psi$ the dual of $\phi$. Assume that $\phi$ tends to plus or minus infinity on $I_1$ and $I_2$, so that the signs are different if $\alpha=\pi$. Then,
		\begin{itemize} \setlength{\leftskip}{-3ex}
			\item $X_{\min}$ admits a vertical segment $L$ over $z_0$, and has
						the line symmetry there.
		\end{itemize}
		Further, the following statements hold at the corresponding part to $L$ under the conjugation and the dual operation.
		\begin{itemize} \setlength{\leftskip}{-3ex}
			\item $X_{\max}$ admits a shrinking singularity, and has the point symmetry there.
			\item $X_{\min}^{\ast}$ admits a horizontal geodesic curvature line,
						and has the planar symmetry there.
			\item $X_{\max}^{\ast}$ admits a null curve as a folding singularity, and has a folded symmetry there.
		\end{itemize}
		Here, $X_{\min}$ and $X_{\max}$ denote the extended surfaces of ${\rm graph}(\phi)$ and ${\rm graph}(\psi)$ across $z_0$, respectively, and $X_{\min}^{\ast}$ and $X_{\max}^{\ast}$ denote their conjugations.
\end{theorem}

\begin{remark}
The curve $\gamma=f^{\ast}((-1,1))$ in the $xy$-plane does not have a point of convexity in the sense used in \cite[Definition 2.6]{BH}. This can be shown in a function theoretical way, by using the fact that $|\widetilde{\omega}|\equiv 1$ where $\widetilde{\omega}$ is the analytic dilatation of $f^{\ast}$. (The concavity of the image of a curve on which $|\omega|\equiv 1$ gets interest, see \cite{BH}, \cite[Section 7.3]{D}, for example.) On the other hand, we can also prove this by a simple causality argument on the maximal surface $X_{\max}^{\ast}$. Indeed, if $\gamma$ is not concave, then we can find two points $z_1, z_2\in \gamma$ such that the straight line joining $z_1$ and $z_2$ lies in the image domain $\Omega^{\ast}=f^{\ast}(\mathbb{D})$. Thus the same argument as in the proof of Proposition \ref{prop:signAlter} leads to a contradiction, since $C=X_{\max}^{\ast}((-1,1))$ is a null curve. 
\end{remark}

\begin{remark}	
Since $X_{\max}^{\ast}$ has folding singularities, the image of $X_{\max}^{\ast}$ can be extended analytically to a timelike minimal surface across the null curve $X_{\max}^{\ast}((-1,1))$, see \cite{Okayama}, \cite{KKSY} and their references.
\end{remark}

%====================Examples=======================

\subsection{Examples}\label{subsec:examples} 
For $n\geq2,\ \alpha=e^{i\pi/n}$ and $r>0$, let $\Omega_n(r)$ be the polygonal domain with vertices $r, \alpha,\ldots, r\alpha^{2k},\alpha^{2k+1},\ldots,r\alpha^{2n-2}, \alpha^{2n-1}$, so that they lie in the boundary in positive orientation in this order. The following examples were given by McDougall-Schaubroeck in \cite{MS}, in order to describe the Jenkins-Serrin  minimal graphs over $\Omega_n(r)$ by using univalent harmonic mappings written as the Poisson integrals of some step functions. Following their construction methods, we can obtain some examples of maximal surfaces simultaneously, as follows.

The first example is the minimal graph diverging to plus or minus infinity alternately on adjacent edges of $\Omega_n(r)$. Let $p\in \mathbb{R}$ be the unique solution of the equation
\begin{equation}
	r\sin\left(\frac{n-1}{n}p\pi\right)=\sin\left(\frac{n-1}{n}(1-p)\pi\right),\ \ \ 0<p<1. \label{eq:p}
\end{equation}
Then, we can see that the Poisson integral $f$ of the step function $\hat{f}$, defined by $\hat{f}(e^{it})=r\alpha^{2k}$ for $(2k-p)\pi/n<t<(2k+p)\pi/n$ and $\hat{f}(e^{it})=\alpha^{2k+1}$ for $(2k+p)\pi/n<t<(2k+2-p)\pi/n$, is univalent and
\begin{equation*}
	f(w)=\frac{r}{\pi}\sum_{k=0}^{n-1} \alpha^{2k}\arg\frac{w-\alpha^{2k}\beta}{w-\alpha^{2k}\overline{\beta}}+\frac{1}{\pi}\sum_{k=0}^{n-1}\alpha^{2k+1}\arg\frac{w-\alpha^{2k+2}\overline{\beta}}{w-\alpha^{2k}\beta},\label{eq:examf}
\end{equation*}
where $\beta=e^{ip\pi/n}$. Further, the analytic dilatation $\omega$ of $f$ is $\omega(w)=w^{2(n-1)}$. Therefore, since $\omega$ has a holomorphic square root $\sqrt{\omega(w)}=w^{n-1}$, the holomorphic  function $F(w)=F_f(w)$ in \eqref{eq:F} can be defined and it holds that
\[
	F(w)=F_f(w)=\frac{1}{\pi}\sin\left(\frac{\pi}{n}\right)\csc\left(\frac{n-1}{n}p\pi\right)\log\frac{w^n-e^{ip\pi}}{w^n-e^{-ip\pi}},
\]
with an appropriate additive constant. Then, we have the dual maximal graph $X_{\max}=(f,{\rm Im}(F))$, together with the desired Jenkins-Serrin minimal graph $X_{\min}=(f,{\rm Re}(F))$, see Figure \ref{zu20}.

\begin{figure}[htbp]
    \begin{tabular}{cc}
      %---- 最初の図 ---------------------------
      \begin{minipage}[t]{0.5\hsize}
        \centering
        \includegraphics[keepaspectratio, scale=0.51]{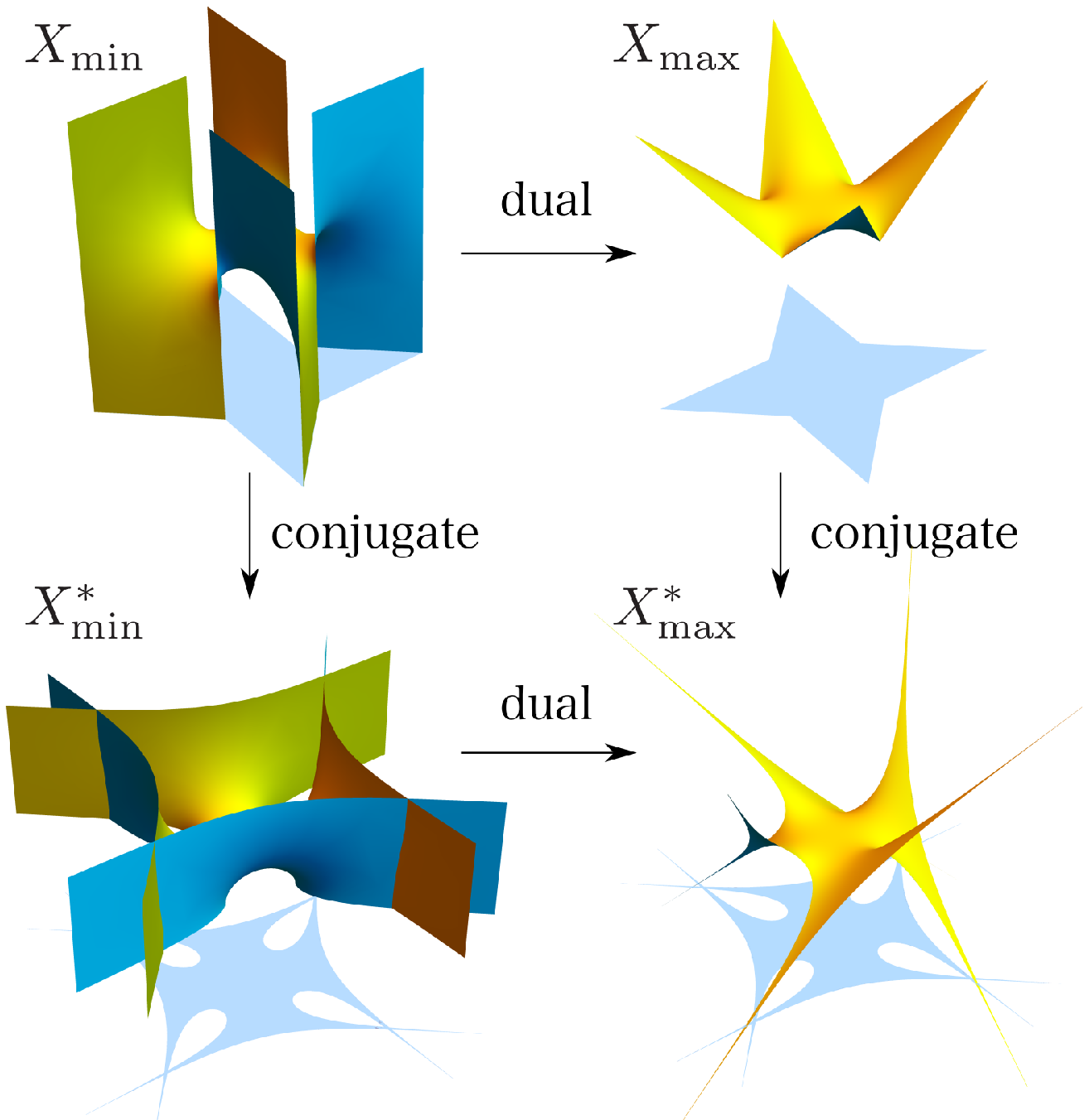}
        \caption{The first example over $\Omega_4(0.4)$, and their conjugations.}
        \label{zu20}
      \end{minipage} &
      %---- 2番目の図 --------------------------
      \begin{minipage}[t]{0.5\hsize}
        \centering
        \includegraphics[keepaspectratio, scale=0.51]{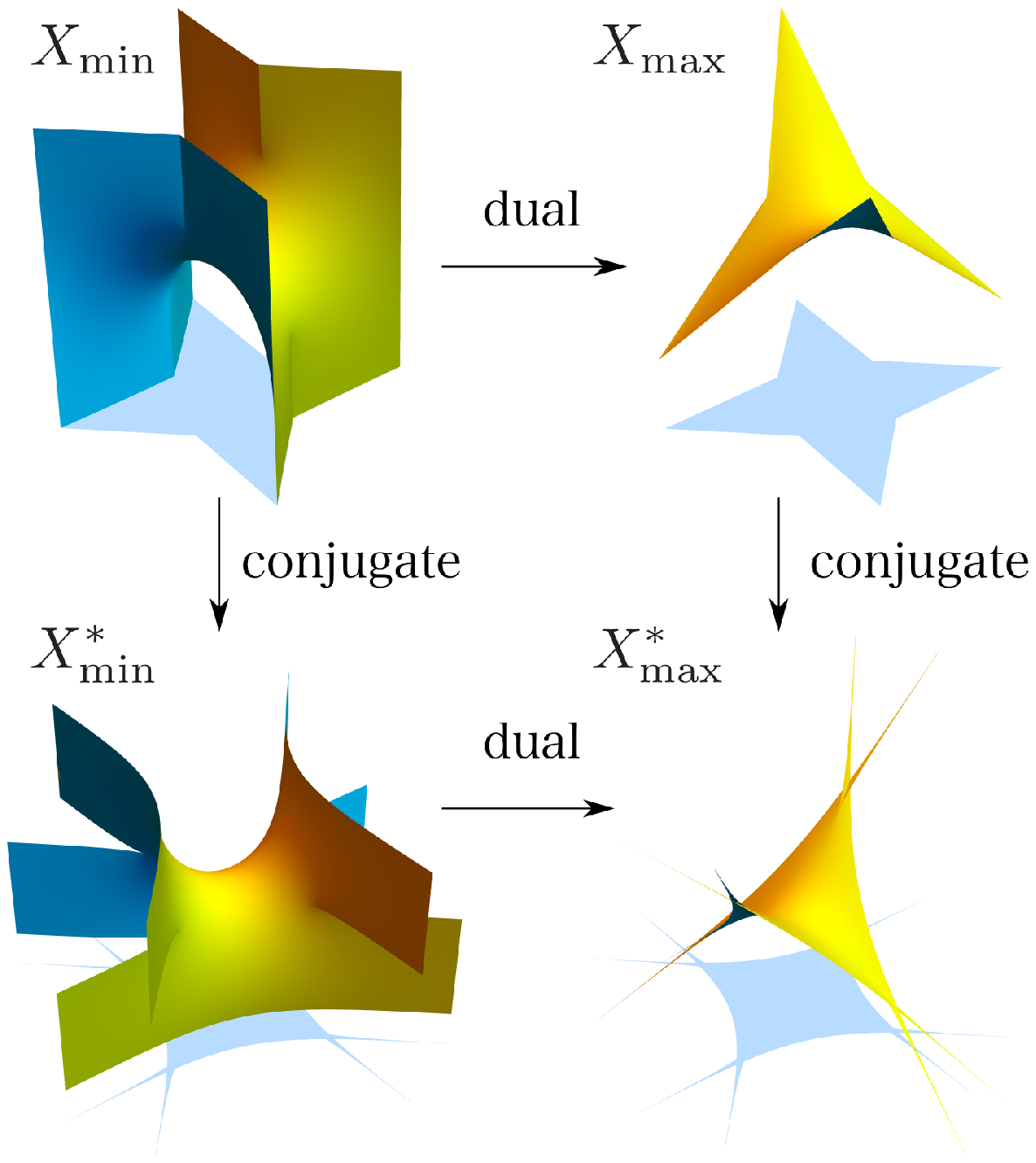}
        \caption{The second example over $\Omega_4(0.4)$, and their conjugations.}
        \label{zu21}
      \end{minipage}
      %---- 図はここまで ----------------------
    \end{tabular}
  \end{figure}

The second example is the minimal graph diverging to plus or minus infinity on each edge, so that the signs change at each convex corner and do not change at each non-convex corner. To construct it, we further suppose that $n\geq 2$ is an even integer and $r<\cos(\pi/n)$. Then the desired minimal graph is given by the same way as in the first example, by using the unique solution $q$ of the equation
\begin{equation}
	r\cos \left(\frac{n-2}{2n}q\pi\right)=\sin\left(\frac{n-2}{2n}(1-q)\pi\right) ,\ \ \ 0<q<1, \label{eq:q}
\end{equation}
instead of $p$. In this case, the harmonic mapping $f$ has an analytic dilatation $\omega(w)=-w^{n-2}$, and thus we have
\[
	F(w)=F_f(w)=\frac{1}{\pi}\sin\left(\frac{\pi}{n}\right)\sec\left(\frac{n-2}{2n}q\pi\right)\log\left(\frac{w^{n/2}-e^{iq\pi/2}}{w^{n/2}+e^{iq\pi/2}}\cdot \frac{w^{n/2}-e^{-iq\pi/2}}{w^{n/2}+e^{-iq\pi/2}}\right).
\]
Similarly, we obtain an explicit representation of the dual maximal graph $X_{\max}=(f,{\rm Im}(F))$ of $X_{\min}=(f,{\rm Re}(F))$, see Figure \ref{zu21}.

Other examples of the Jenkins-Serrin minimal graphs over $\Omega_n (r)$ are also constructed in \cite[Section 5]{MS}. See their article also for detailed discussions on the sign changing properties of minimal graphs and on the meanings of the equations \eqref{eq:p} and \eqref{eq:q}.

%====================Reflection property for lightlike boundary lines=======================

\subsection{Reflection property along lightlike boundary lines of maximal surfaces}\label{subsec:lightlike_reflection}
As well as the case of minimal surfaces, the reflection principle holds for maximal surfaces, that is, if a  maximal surface contains a spacelike straight line segment $l$, then the surface is invariant under the line symmetry with respect to $l$. 

However, as far as the authors know, it is not known whether such a reflection property is valid for lightlike lines on boundaries of maximal surfaces. (The same question was also raised for timelike minimal surfaces in \cite[p.~1095]{KKSY}.)
Although, many known examples of maximal surfaces have planar symmetries along lightlike lines (cf.~\cite{CR}, \cite{K}, for example), we cannot expect such a symmetry in general. In fact, the following family of maximal surfaces
\[
\mathcal{S}_p=\{ p^2\cos{(qx)}+q^2\cos{(py)}=\cos{(pqt)}\},\quad p^2+q^2=1,\ p,q>0
\]
have lightlike lines, along which the surface have neither the line symmetry nor planar symmetry except the case $p=1/\sqrt{2}$. See Figure \ref{fig:3}. 

\begin{figure}[htb]
 \begin{center}
   \begin{tabular}{c@{\hspace{0.3cm}}c@{\hspace{-0.4cm}}c}
        \includegraphics[height=6.0cm]{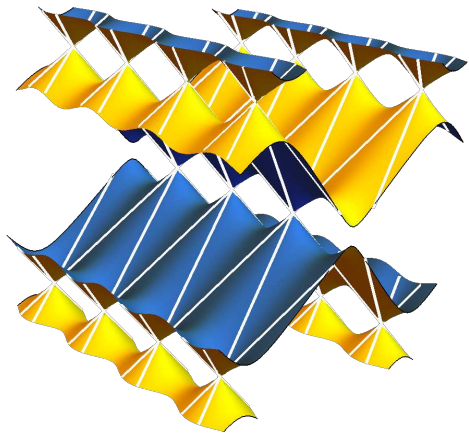} 
    \end{tabular}
\caption{The surface $\mathcal{S}_{1/3}$ containing lightlike lines along which the surface have neither the line symmetry nor planar symmetry.} 
\label{fig:3}
 \end{center}
\end{figure}

In contrast, lightlike lines induce a point symmetry, not a line or a planar symmetry in many situations as follows.

\begin{corollary}\label{cor:LineSymmetry}
If a maximal graph tamely degenerates to adjacent two lightlike line segments, then the intersection is a shrinking singularity of the extended generalized maximal surface in the sense of Theorem $\ref{thm:reflection}$, and the extended surface has the point symmetry with respect to this shrinking singularity. 
\end{corollary}

Indeed, the above situation in Corollary \ref{cor:LineSymmetry} occurs in the surface $S_p$ and all of the maximal surfaces in \cite{CR}, in addition to the duals of the Jenkins-Serrin minimal graphs.

%=========================ACKNOWLEDGEMENT=======================================

\begin{acknowledgement}
The authors would like to express their gratitude to Professor Daoud Bshouty for giving us important and helpful information on related works. 
They also would like to thank Professors Masaaki Umehara and Kotaro Yamada, who taught them the family $\{\mathcal{S}_p\}$ of maximal surfaces in Section \ref{subsec:lightlike_reflection}.

\end{acknowledgement}

%================================REFERENCES=======================================
%\bibliographystyle{amsxport}  \bibliography{bib/databaseAF2} %bibtex用   最終的にはbblファイルをコピペする.


\end{document}